\documentclass[a4paper]{amsart}
\usepackage{wrapfig}
\usepackage[dvips]{graphicx}
\usepackage{amsmath,amsthm,amssymb,amscd}
\usepackage{mathrsfs}
\usepackage[all]{xy}
\theoremstyle{definition}
\newtheorem{thm}{Theorem}[section]
\newtheorem{Def}[thm]{Definition}
\newtheorem{pro}[thm]{Proposition}
\newtheorem{cor}[thm]{Corollary}
\newtheorem{lem}[thm]{Lemma}

\newtheorem*{mainthm}{Theorem}
\theoremstyle{definition}

\begin{document}
\title[Finite abelian group actions on the Razak-Jacelon algebra]{Approximate representability of 
finite abelian group actions on the Razak-Jacelon algebra}
\author{Norio Nawata}
\address{Department of Pure and Applied Mathematics, Graduate School of Information Science 
and Technology, Osaka University, Yamadaoka 1-5, Suita, Osaka 565-0871, Japan}
\email{nawata@ist.osaka-u.ac.jp}
\thanks{This work was supported by JSPS KAKENHI Grant Number 20K03630}

\begin{abstract}
Let $A$ be a simple separable nuclear monotracial C$^*$-algebra, and let $\alpha$ be 
an  outer action of a finite abelian group $\Gamma$ on $A$. 
In this paper, we show that $\alpha\otimes \mathrm{id}_{\mathcal{W}}$ on $A\otimes\mathcal{W}$ 
is approximately representable if and only if the characteristic invariant of $\tilde{\alpha}$ is trivial, where 
$\mathcal{W}$ is the Razak-Jacelon algebra and $\tilde{\alpha}$ is the induced action 
on the injective II$_1$ factor $\pi_{\tau_{A}}(A)^{''}$.  As an application of this result, we  
classify such actions up to conjugacy and cocycle conjugacy. 
We also construct the model actions. 
\end{abstract}
\maketitle

\section*{Introduction}

In \cite{Jones}, Jones gave a complete classification of finite group actions on the injective 
II$_1$ factor up to conjugacy. This can be regarded as a generalization of 
Connes' classification \cite{C3} of periodic automorphisms of the injective II$_1$ factor. 
In this paper, we study a C$^*$-analog of these results. 

There exist some difficulties for the classification of (amenable) group actions on 
``classifiable'' C$^*$-algebras 
because of $K$-theoretical obstructions. We refer the reader to \cite{I} for details. 
In spite of these difficulties, Gabe and Szab\'o classified outer actions of countable discrete 
amenable groups on Kirchberg algebras up to cocycle conjugacy by  equivariant $KK$-theory in 
\cite{GS}. This classification can be regarded as a C$^*$-analog in Kirchberg algebras of Ocneanu's 
classification theorem \cite{Oc}. 
For stably finite classifiable C$^*$-algebras, such a classification is an interesting open problem. 

The Razak-Jacelon algebra $\mathcal{W}$ (\cite{J} and \cite{Raz}) is the simple separable 
nuclear monotracial $\mathcal{Z}$-stable C$^*$-algebra which is $KK$-equivalent to $\{0\}$. 
We regard $\mathcal{W}$ as a monotracial analog of the Cuntz algebra $\mathcal{O}_2$. 
Indeed, if $A$ is a simple separable nuclear monotracial C$^*$-algebra, then $A\otimes\mathcal{W}$ 
is isomorphic to $\mathcal{W}$ by classification results in \cite{CE} and \cite{EGLN} 
(see also \cite{Na4}). Let $\alpha$ be an action on a simple separable nuclear monotracial 
C$^*$-algebra $A$. Then the  action $\alpha\otimes\mathrm{id}_{\mathcal{W}}$ on 
$A\otimes\mathcal{W}$ is an action on $\mathcal{W}$. 
We call such a tensor product type action a \textit{$\mathcal{W}$-type action}. 
Note that there exist no 
$K$-theoretical obstructions for $\mathcal{W}$-type actions. 
Therefore we can recognize difficulties due to stably finiteness by studying 
$\mathcal{W}$-type actions. 
In \cite{Na5}, the author showed that if $\alpha$ is a strongly outer action of 
a countable discrete amenable group, then $\mathcal{W}$-type actions are unique up to 
cocycle conjugacy. 
(Note that an action $\alpha$ of a discrete countable group on a simple separable 
monotracial C$^*$-algebra $A$ is \textit{strongly outer} if the induce action $\tilde{\alpha}$ by 
$\alpha$ on $\pi_{\tau_A}(A)^{''}$ is outer, where $\pi_{\tau_A}$ is 
the Gelfand-Naimark-Segal (GNS) representation associated with the unique tracial state $\tau_A$ 
on $A$.) 
This result can be regarded as a monotracial analog of Szab\'o's 
equivariant Kirchberg-Phillips type absorption theorem for $\mathcal{O}_2$ \cite{Sza4}.  
In this paper, we study $\mathcal{W}$-type outer actions of finite abelian groups, 
which include ``weakly inner'' (or non-strongly outer) actions. 

One of the main results in this paper is the characterization of approximate representability
of $\alpha\otimes\mathrm{id}_{\mathcal{W}}$ by using the characteristic invariant of $\tilde{\alpha}$.
Note that approximate representability is the dual notion of the Rohlin 
property \cite{I1} (see also \cite{GHS}, \cite{GSan1}, \cite{Na0} and \cite{San}). 
Also, we can classify approximately representable actions by using classification results of 
C$^*$-algebras (see \cite{GSan2}, \cite{I1} and \cite{Na0}). 
As an application of this result, we can classify such actions up to conjugacy and cocycle conjugacy. 
For an action $\alpha$ on a simple separable monotracial C$^*$-algebra $A$ of a countable
discrete group $\Gamma$, put 
$N(\tilde{\alpha}):=\{g\in \Gamma \; | \; \tilde{\alpha}_g=\mathrm{Ad}(u)\; \text{for some}\; u\in 
\pi_{\tau_A}(A)^{''}\}$ and let $i(\tilde{\alpha})$ be the inner invariant of $\tilde{\alpha}$ defined in 
\cite{Jones}. We show the following classification result in this paper. 

\begin{mainthm}
(Corollary \ref{main:cor}) \ \\
Let $A$ and $B$ be 
simple separable nuclear monotracial C$^*$-algebras, and let $\alpha$ and $\beta$ be 
outer actions of a finite abelian group $\Gamma$ on $A$ and $B$, respectively. 
Assume that the characteristic invariants of $\tilde{\alpha}$ and $\tilde{\beta}$ are trivial. 
Then \ \\
(i) $\alpha\otimes \mathrm{id}_{\mathcal{W}}$ on $A\otimes\mathcal{W}$ and 
$\beta\otimes \mathrm{id}_{\mathcal{W}}$ on $B\otimes\mathcal{W}$ are cocycle conjugate 
if and only if $N(\tilde{\alpha})=N(\tilde{\beta})$; \ \\
(ii) $\alpha\otimes \mathrm{id}_{\mathcal{W}}$ on $A\otimes\mathcal{W}$ and 
$\beta\otimes \mathrm{id}_{\mathcal{W}}$ on $B\otimes\mathcal{W}$ are conjugate if and only if 
$N(\tilde{\alpha})=N(\tilde{\beta})$ and $i(\tilde{\alpha})=i(\tilde{\beta})$.
\end{mainthm}

To the author's best knowledge, this classification is the first abstract classification result 
for ``weakly inner''  (or non-strongly outer) actions on stably finite C$^*$-algebras without inductive 
limit type structures. 
Also, we construct the model actions of the actions in the theorem above. 
This construction might be of independent interest. 

\section{Preliminaries}\label{sec:pre}

\subsection{Group actions}

Let  $\alpha$ and $\beta$ be actions of a discrete group $\Gamma$ on C$^*$-algebras $A$ and 
$B$, respectively. 
We say that $\alpha$ is \textit{conjugate} to $\beta$ if there exists an isomorphism $\theta$ from $A$ 
onto $B$ such that $\theta\circ \alpha_g= \beta_g\circ \theta$ for any $g\in\Gamma$. 
An \textit{$\alpha$-cocycle} is a map $u$ from $\Gamma$ to 
the unitary group of the multiplier algebra $M(A)$ of $A$ such that $u_{gh}=u_g\alpha_g(u_h)$. 
Note that we denote the induced action on $M(A)$ by the same symbol $\alpha$ for simplicity. 
We say that 
$\alpha$ is \textit{cocycle conjugate} to $\beta$ if there exist an isomorphism $\theta$ from $A$ to 
$B$ and a $\beta$-cocycle $u$ such that $\theta\circ \alpha_g=\mathrm{Ad}(u_g)\circ \beta_g \circ 
\theta$ for any $g\in \Gamma$. 
Let 
$$
N(\alpha):= \{g\in\Gamma \; |\; \alpha_g=\mathrm{Ad}(u)\; \text{for some}\; u\in M(A)\}.
$$
It is said to be that $\alpha$ is \textit{outer} if $N(\alpha)=\{\iota\}$ where $\iota$ is 
the identity element in $\Gamma$. 

We denote by $A^{\alpha}$ and $A\rtimes_{\alpha}\Gamma$ the fixed point subalgebra and 
the reduced crossed product C$^*$-algebra, respectively. 
Let $E_{\alpha}$ denote the canonical conditional expectation from $A\rtimes_{\alpha}\Gamma$ onto 
$A$. 
If $A$ is simple and $\alpha$ is outer, then  $A\rtimes_{\alpha}\Gamma$ is simple by \cite{K}. 
Assume that $\Gamma$ is a finite abelian group. Let 
$$
e_{\alpha}:=
\frac{1}{|\Gamma|}\sum_{g\in\Gamma}\lambda_g\in M(A\rtimes_{\alpha}\Gamma)
$$ 
where $\lambda_g$ is the implementing unitary of $\alpha_g$ in $M(A\rtimes_{\alpha}\Gamma)$ 
and $|\cdot|$ denotes cardinality. 
Then $e_{\alpha}$ is a projection and 
$e_{\alpha}(A\rtimes_{\alpha}\Gamma)e_{\alpha}$ is isomorphic to $A^{\alpha}$. 
We denote by $\hat{\alpha}$ the dual action of $\alpha$, that is, 
$\hat{\alpha}_{\eta}(\sum_{g\in \Gamma}a_g\lambda_g)=
\sum_{g\in \Gamma}\eta(g)a_g\lambda_g$ for any 
$\sum_{g\in \Gamma}a_g\lambda_g\in A\rtimes_{\alpha}\Gamma$ and $\eta\in \hat{\Gamma}$. 
Note that $\hat{\Gamma}$ is isomorphic to $\Gamma$ since we assume that 
$\Gamma$ is a finite abelian group. 

Let $T_1(A)$ denote the tracial state space of $A$. 
Every tracial state on $A$ can be uniquely extended to a tracial state on $M(A)$. We denote it by the 
same symbol for simplicity. If $\varphi$ is a nondegenerate homomorphism from $A$ to $B$, then 
$\varphi$ induces an affine continuous map $T(\varphi)$ from $T_1(B)$ to $T_1(A)$ by 
$T(\varphi)(\tau)= \tau\circ \varphi$ for any $\tau\in T_1(B)$. 
Hence every action $\alpha$ on $A$ induces an action $T(\alpha)$ on $T_1(A)$. 
We denote by $T_1(A)^{\alpha}$ the fixed point set of this induced action. 
Straightforward arguments show the following proposition. 

\begin{pro}\label{pro:conjugacy-trace-spaces}
Let  $\alpha$ and $\beta$ be actions of a finite abelian group $\Gamma$ on C$^*$-algebras $A$ and 
$B$, respectively. Then \ \\
(i) if $\alpha$ and $\beta$ are cocycle conjugate, then there exists an affine homeomorphism $F$
from $T_1(B\rtimes_{\beta}\Gamma)$ onto $T_1(A\rtimes_{\alpha}\Gamma)$ such that 
$F\circ T(\hat{\beta}_{\eta})=T(\hat{\alpha}_{\eta})\circ F$ for any $\eta\in \hat{\Gamma}$; \ \\
(ii) if $\alpha$ and $\beta$ are conjugate, then there exists an affine homeomorphism $F$
from $T_1(B\rtimes_{\beta}\Gamma)$ onto $T_1(A\rtimes_{\alpha}\Gamma)$ such that 
$F(\tau)(e_{\alpha})=\tau (e_\beta)$ for any $\tau \in T_1(B\rtimes_{\beta}\Gamma)$ and 
$F\circ T(\hat{\beta}_{\eta})=T(\hat{\alpha}_{\eta})\circ F$ for any $\eta\in \hat{\Gamma}$.
\end{pro}

Let $\tau$ be an $\alpha$-invariant tracial state on $A$, that is, $\tau\in T_1(A)^{\alpha}$. 
Then $\alpha$ induces an action $\tilde{\alpha}$ on $\pi_{\tau}(A)^{''}$ where 
$\pi_{\tau}$ is the GNS representation associated with $\tau$. 

\subsection{Finite abelian group actions on the injective II$_1$ factor}

We shall recall some results in \cite{Jones}. We refer the reader to 
\cite{Jones}, \cite{Jones2} and \cite{Oc} for details. 
Let $M$ be a II$_1$ factor, and let $\delta$ be an action of a finite abelian group $\Gamma$ 
on $M$. Although we can consider in a more general setting, 
we assume that $\Gamma$ is a finite abelian group for simplicity. 
Also, note that the von Neumann algebraic crossed product of 
$(M, \Gamma, \delta)$ is isomorphic to (the reduced crossed product C$^*$-algebra) 
$M\rtimes_{\delta}\Gamma$ by this assumption. 
By definition of $N(\delta)$, there exists a map $v$ from $N(\delta)$ to the unitary group of 
$M$ such that $\delta_h=\mathrm{Ad}(v_h)$ for any $h\in N(\delta)$. 
For any $g\in \Gamma$ and $h\in N(\delta)$, there exists  a complex number 
$\lambda_{\delta}(g, h)$ with $|\lambda_{\delta} (g, h)|=1$ such that 
$\delta_g(v_h)=\lambda_{\delta}(g,h)v_h$. It is easy to see that 
$\lambda_{\delta}(g, h)$ does not depend on the choice of $v$. 
We say that \textit{the characteristic invariant of $\delta$ is trivial} if 
$\lambda_{\delta}(g,h)=1$ for any $g\in \Gamma$ and $h\in N(\delta)$. 
We refer the reader to \cite[Section 1.2]{Jones} for the precise definition of 
the characteristic invariant. 
Note that we may assume that $v$ is a unitary representation since 
$N(\delta)$ is a finite abelian group. 
Indeed, if $N(\delta)$ is a cyclic group generated by $g$ of order $n$, then there exists a 
complex number $\gamma$ with $|\gamma| =1$ such that $v_g^n=\gamma 1$. 
Choose an $n$-th root $\gamma^{\prime}$ of $\gamma$ and define a map $v^{\prime}$ from 
$N(\delta)$ to the unitary group of $M$ by $v^{\prime}_{g^k}:= \gamma^{\prime k}v_g^k$ for any 
$1\leq k\leq n$. Then $v^{\prime}$ is a unitary representation such that 
$\delta_h=\mathrm{Ad}(v_h^{\prime})$ for any $h\in N(\delta)$. 
Since every finite abelian group is a finite direct sum of cyclic groups of finite order, 
if $N(\delta)$ is a finite abelian group, then there exists such a unitary representation. 

\begin{pro}\label{pro:non-trivial-characteristic}
Let $M$ be a II$_1$ factor, and let $\delta$ be an action of a finite abelian group $\Gamma$ 
on $M$. If the characteristic invariant of $\delta$ is not trivial, then the dual action $\hat{\delta}$ 
is not outer. 
\end{pro}
\begin{proof}
Since the characteristic invariant of $\delta$ is not trivial, there exist $g_0\in\Gamma$ and 
$h_0\in N(\delta)$ such that $\lambda_{\delta}(g_0, h_0)\neq 1$. 
Define a map $\eta_0$ from $\Gamma$ to $\mathbb{C}$ by $\eta_0(g):=\lambda_{\delta} (g, h_0)$ 
for any $g\in\Gamma$. Then $\eta_0$ is a nontrivial character on $\Gamma$. 
Let $v_{h_0}$ be a unitary element in $M$ such that $\delta_{h_0}=\mathrm{Ad}(v_{h_0})$. 
Then we have 
\begin{align*}
\mathrm{Ad}(\lambda_{h_0}v_{h_0}^*)\left(\sum_{g\in \Gamma}a_g\lambda_g\right)
&=\sum_{g\in\Gamma}\lambda_{h_0}v_{h_0}^*a_g \lambda_gv_{h_0}\lambda_{h_0}^* \\
&= \sum_{g\in\Gamma}\lambda_{h_0}v_{h_0}^*a_g \delta_{g}(
v_{h_0})\lambda_g\lambda_{h_0}^* \\
&=\sum_{g\in\Gamma}\lambda_{\delta}(g,h_0) \lambda_{h_0}v_{h_0}^*a_gv_{h_0}\lambda_{h_0}^*
\lambda_{g} \\
&=\sum_{g\in\Gamma}\lambda_{\delta}(g,h_0) \lambda_{h_0}\delta_{h^{-1}_0}(a_g)\lambda_{h_0}^*
\lambda_{g} \\
&= \sum_{g\in\Gamma}\eta_0(g)a_g\lambda_{g} \\
&= \hat{\delta}_{\eta_0}\left( \sum_{g\in \Gamma}a_g\lambda\right)
\end{align*}
for any $\sum_{g\in \Gamma}a_g\lambda_g \in M\rtimes_{\gamma}\Gamma$. 
Therefore $\hat{\delta}$ is not outer. 
\end{proof}

In the rest of this section, we assume that the characteristic invariant of $\delta$ is trivial, and 
$v$ is a unitary representation $v$ of $N(\delta)$ on 
$M$ such that $\delta_h=\mathrm{Ad}(v_h)$ for any $h\in N(\delta)$. 
Define a homomorphism $\Phi_v$ from the group algebra $\mathbb{C}N(\delta)$ to $M$ 
by $\Phi_v(\sum_{h\in N(\delta)}c_h h)=\sum_{h\in N(\delta)}c_hv_h$. 
Let $P_{N(\delta)}$ be the set of minimal projections in  $\mathbb{C}N(\delta)$, and let 
$M(P_{N(\delta)})$ be the set of probability measures on $P_{N(\delta)}$. 
Note that we have $P_{N(\delta)}=\left\{\frac{1}{|N(\delta)|}\sum_{h\in N(\delta)} \eta (h)h\; |\; \eta\in 
\hat{N(\delta)}\right\}$.
Define a probability measure $m_{v}(\delta)$ on $P_{N(\delta)}$ by $\tau_{M}(\Phi_v(p))$ for any 
$p\in P_{N(\delta)}$ where $\tau_{M}$ is the unique tracial state on $M$. 
Note that $m_v(\delta)$ depends on the choice of $v$. 
For any character $\eta$ of $N(\delta)$, define an automorphism $\partial(\eta)$ of 
$\mathbb{C}N(\delta)$ by $\partial(\eta) (\sum_{h\in N(\delta)}c_h h)=
\sum_{h\in N(\delta)}\eta(h)c_h h$. 
Define an equivalence relation $\sim$ on $M(P_{N(\delta)})$ by $m\sim m^{\prime}$ if there exists a 
character $\eta$ of $N(\delta)$ such that  $m = m^{\prime }\circ \partial(\eta)$. 
Put $i(\delta):= [m_v(\delta)]\in M(P_{N(\delta)})/\sim$. Then $i(\delta)$ does not depend on the choice 
of $v$ and $i(\delta)$ is a conjugacy invariant. The following theorem is a part of 
Ocneanu's classification theorem and Jones' classification theorem. 

\begin{thm}\label{thm:jones}
(Cf. \cite[Theorem 2.6]{Oc} and \cite[Theorem 1.4.8]{Jones}) \ \\
Let $\delta$ and $\delta^{\prime}$ be actions of a finite abelian group $\Gamma$ on the 
injective II$_1$ factor $M$. Assume that the characteristic invariants of $\delta$ and 
$\delta^{\prime}$ are trivial. Then \ \\ 
(i) $\delta$ and $\delta^{\prime}$ are cocycle conjugate if and only if 
$N(\delta)=N(\delta^{\prime})$; \ \\
(ii) $\delta$ and $\delta^{\prime}$ are conjugate if and only if 
$N(\delta)=N(\delta^{\prime})$ and $i(\delta)=i(\delta^{\prime})$.
\end{thm}

Define a map $\Pi_v$ from $\mathbb{C}N(\delta)$ to $M\rtimes_{\delta}\Gamma$ 
by $\Pi_v(\sum_{h\in N(\delta)}c_h h)=\sum_{h\in N(\delta)}c_h v_h\lambda_h^*$. 
Then $\Pi_v$ is an isomorphism from $\mathbb{C}N(\delta)$ onto 
the center $Z(M\rtimes_{\delta}\Gamma)$ of $M\rtimes_{\delta}\Gamma$ 
by \cite[Corollary 2.2.2]{Jones}. 
This implies that $T_1(M\rtimes_{\delta}\Gamma)$ is an $|N(\delta)|$-simplex. 
Indeed, for any $p\in P_{N(\delta)}$, define a tracial state $\tau_{p}$ on 
$M\rtimes_{\delta}\Gamma$ by $\tau_{p}(x):= |N(\delta)|\tau_{M}\circ E_{\delta}(\Pi_v(p)x)$
for any $x\in M\rtimes_{\delta}\Gamma$. 
Then $\tau_{p}$ is a unique tracial state on $\Pi_v(p)M\rtimes_{\delta}\Gamma$ since 
$\Pi_v(p)M\rtimes_{\delta}\Gamma$ is a factor. 
If $\tau$ is a tracial state on $M\rtimes_{\delta}\Gamma$, then 
$\tau(x)=\sum_{p\in P_{N(\delta)}}\tau (\Pi_v(p)x)=\sum_{p\in P_{N(\delta)}}\tau (\Pi_v(p))\tau_{p}(x)$ 
for any $x\in M\rtimes_{\delta}\Gamma$. 
Hence $T_1(M\rtimes_{\delta}\Gamma)$ is an $|N(\delta)|$-simplex and 
the set of extremal tracial states on $M\rtimes_{\delta}\Gamma$ is equal to 
$\{\tau_{p}\; |\; p\in P_{N(\delta)}\}$. 
Easy computations show that we have 
$\tau_{p}(e_{\delta})=\frac{|N(\delta)|}{|\Gamma|}\tau_{M}(\Phi_v(p))$ for any 
$p\in P_{N(\delta)}$. Therefore we can recover $i(\delta)$ by considering the extremal tracial states 
on $M\rtimes_{\delta}\Gamma$. 

\subsection{Finite abelian group actions on monotracial C$^*$-algebras}

We say a C$^*$-algebra $A$ is \textit{monotracial} if $A$ has a unique tracial state and no 
unbounded traces. For a monotracial C$^*$-algebra $A$, 
we denote by $\tau_{A}$ the unique tracial state on $A$ unless otherwise specified.
Let $A$ be a simple separable monotracial C$^*$-algebra, and let $\alpha$ be 
an outer action of a finite abelian group $\Gamma$ on $A$. 
Then $\pi_{\tau_{A}}(A)^{''}$ is a II$_1$ factor. 
Note that $A$ is not of type I since $A$ has an outer action. 
Of course, $N(\tilde{\alpha})$ and $i(\tilde{\alpha})$ are a cocycle conjugacy invariant and 
a conjugacy invariant for $\alpha$ on $A$, respectively. 
Since we assume that $A$ is simple, $\tau_A\circ E_{\alpha}$ is faithful. 
Hence 
we can regard $A\rtimes_{\alpha}\Gamma$ and 
$M(A\rtimes_{\alpha}\Gamma)$ as subalgebras of 
$\pi_{\tau_A\circ E_{\alpha}}(A\rtimes_{\alpha}\Gamma)^{''}\cong 
\pi_{\tau_{A}}(A)^{''}\rtimes_{\tilde{\alpha}}\Gamma$.

\begin{lem}\label{lem:trace-spaces-crossed-products}
Let $B$ be a simple separable C$^*$-algebra with a compact tracial state space $T_1(B)$, and let 
$\beta$ be an action of a finite group $\Gamma$ on $B$ with $T_1(B)^{\beta}=\{\tau_{0}\}$.  
Assume that $\pi_{\tau_{0}}(B)^{''}$ has finitely many extremal tracial states (this is equivalent to that 
every tracial state on $\pi_{\tau_{0}}(B)^{''}$ is normal). 
Then the restriction map $T_1(\pi_{\tau_{0}}(B)^{''})\ni \tau \mapsto \tau|_B\in T_1(B)$ is an 
affine homeomorphism. 
\end{lem}
\begin{proof}
It is obvious that the restriction map is a continuous affine map. 
Since $\pi_{\tau_{0}}(B)$ is weakly dense in $\pi_{\tau_{0}}(B)^{''}$ and every tracial state on 
$\pi_{\tau_{0}}(B)^{''}$ is normal, the restriction map is injective. 
We shall show the surjectivity. 
Let $\{\tau_1,\tau_2,...,\tau_{k}\}$ be the set of extremal tracial states 
on $\pi_{\tau_{0}}(B)^{''}$. 
Note that $\tau_{1}|_B$, $\tau_{2}|_B$,..., $\tau_{k}|_B$  are extremal tracial states on $B$ 
because $\tau$ is an extremal tracial state if and only if $\tau$ is a 
factorial tracial state. 
Since $T_1(B)$ is compact, it is enough to show that the set of extremal tracial states on $B$ is equal to 
$\{\tau_1|_B,\tau_2|_B,...,\tau_{k}|_B\}$ by the Krein-Milman theorem. 
On the contrary, suppose that there were an extremal tracial state $\sigma$ on $B$ such that 
$\sigma\notin \{\tau_1|_B,\tau_2|_B,...,\tau_{k}|_B\}$.
Since $\tau_0$ is $\beta$-invariant, $\beta$ induces an action $\tilde{\beta}$ on 
$\pi_{\tau_{0}}(B)^{''}$. It is easy to  see that $\{\tau_1,\tau_2,...,\tau_{k}\}$ is 
a $\tilde{\beta}$-invariant set. Hence $\{\tau_1|_B,\tau_2|_B,...,\tau_{k}|_B\}$ is a $\beta$-invariant set 
and $\sigma\circ \beta_g\notin \{\tau_1|_B,\tau_2|_B,...,\tau_{k}|_B\}$ for any $g\in \Gamma$. 
Since $\tau_0$ is the unique $\beta$-invariant tracial state on $B$, 
$$
\tau_0=\frac{1}{|\Gamma|}\sum_{g\in\Gamma}\tau_1\circ \beta_g=
\frac{1}{|\Gamma|}\sum_{g\in\Gamma}\sigma \circ \beta_g .
$$
Since $T_1(B)$ is a Choquet simplex, which we remind the reader requires $\tau_0$ to have a unique 
representation as a convex combination of the finitely many extremal traces of $T_1(B)$, this is 
a contradiction. 
\end{proof}

\begin{pro}\label{pro:trace-spaces-crossed-products}
Let $A$ be a simple separable monotracial C$^*$-algebra, and let $\alpha$ be 
an outer action of a finite abelian group $\Gamma$ on $A$. 
Then the restriction map $T_1(\pi_{\tau_A\circ E_{\alpha}}(A\rtimes_{\alpha}\Gamma)^{''})\ni 
\tau \mapsto 
\tau|_{A\rtimes_{\alpha}\Gamma}\in T_1(A\rtimes_{\alpha}\Gamma)$ is an affine homeomorphism. 
\end{pro}
\begin{proof}
By the outerness of $\alpha$, $A\rtimes_{\alpha}\Gamma$ is simple. 
\cite[Corollary 2.2.3]{Jones} implies that 
$\pi_{\tau_A\circ E_{\alpha}}(A\rtimes_{\alpha}\Gamma)^{''}\cong 
\pi_{\tau_{A}}(A)^{''}\rtimes_{\tilde{\alpha}}\Gamma$ 
has finitely many extremal tracial states. 
Since $A$ is monotracial, \cite[Lemma 2.2]{Na0} implies that $T_1(A\rtimes_{\alpha}\Gamma)$ 
is compact. Furthermore, we see that $T_1(A\rtimes_{\alpha}\Gamma)^{\hat{\alpha}}$ 
is a one point set by the Takesaki-Takai duality theorem. 
Applying Lemma \ref{lem:trace-spaces-crossed-products} to 
$B=A\rtimes_{\alpha}\Gamma$ and $\beta=\hat{\alpha}$, 
we obtain the conclusion. 
\end{proof}

The following corollary is an immediate consequence of the proposition above and the previous subsection. 

\begin{cor}
Let $A$ be a simple separable monotracial C$^*$-algebra, and let $\alpha$ be 
an outer action of a finite abelian group $\Gamma$ on $A$. 
Assume that the characteristic invariant of $\tilde{\alpha}$ is trivial. Then 
$T_1(A\rtimes_{\alpha}\Gamma)$ is an $|N(\tilde{\alpha})|$-simplex. 
\end{cor}

A probability measure $m$ on a finite set $P$ is said to have 
\textit{full support} if $m(p)>0$ for any $p\in P$. 

\begin{pro}\label{pro:realized-invariant}
Let $A$ be a simple separable monotracial C$^*$-algebra, and let $\alpha$ be 
an outer action of a finite abelian group $\Gamma$ on $A$. 
Assume that the characteristic invariant of $\tilde{\alpha}$ is trivial. 
If $m$ is a probability measure on $P_{N(\tilde{\alpha})}$ such that 
$i(\tilde{\alpha})=[m]$, then $m$ has full support. 
\end{pro}
\begin{proof}
By Proposition \ref{pro:trace-spaces-crossed-products} and the previous subsection, we may assume 
that 
$m(p)=\frac{|\Gamma|}{|N(\delta)|}\tau_{p}(e_{\alpha})$ for any $p\in P_{N(\tilde{\alpha})}$ 
where $\tau_{p}$ is the extremal tracial state on 
$A\rtimes_{\alpha}\Gamma$ corresponding to $p$. 
Since $\alpha$ is outer, $A\rtimes_{\alpha}\Gamma$ is simple. Hence 
we have $\tau_{p}(e_{\alpha})>0$ for any  $p\in P_{N(\tilde{\alpha})}$ because 
$\tau_{p}$ is faithful and $e_{\alpha}$ is a non-zero projection in 
$M(A\rtimes_{\alpha}\Gamma)$. Therefore we obtain the conclusion. 
\end{proof}

\subsection{Kirchberg's relative central sequence C$^*$-algebras} 

Fix a free ultrafilter $\omega$ on $\mathbb{N}$. Let $A$ and $B$ be C$^*$-algebras, and let 
$\Phi$ be a homomorphism from $A$ to $B$. Put 
$$
B^{\omega}:=\ell^{\infty}(\mathbb{N}, B)/\{\{x_n\}_{n\in\mathbb{N}}\in \ell^{\infty}(\mathbb{N}, B)\; 
|\;  \lim_{n\to\omega} \|x_n\|=0\} 
$$
and we regard $B$ as a C$^*$-subalgebra of $B^{\omega}$ consisting of equivalence classes of 
constant sequences.  
We denote by $(x_n)_n$ a representative of an element in $B^{\omega}$. 
Set 
$$
F(\Phi (A), B)=B^{\omega}\cap \Phi(A)^{\prime}/ \{(x_n)_n\in B^{\omega}\cap \Phi(A)^{\prime}\; |
\; (x_n \Phi(a))_n=0\; \text{for any}\; a\in A\} 
$$
and we call it \textit{Kirchberg's relative central sequence C$^*$-algebra}. 
If $A=B$ and $\Phi=\mathrm{id}_A$, then we denote $F(\Phi(A), B)$ by $F(B)$. 
Every action $\alpha$ of a discrete group $\Gamma$ on $B$ with 
$\alpha_g(\Phi(A))=\Phi(A)$ for any $g\in\Gamma$ induces an action on $F(\Phi(A), B)$. 
We denote it by the same symbol $\alpha$ for simplicity unless otherwise specified. 

Let $A$ be a simple separable non-type I nuclear monotracial C$^*$-algebra and $B$ a 
monotracial C$^*$-algebra with strict comparison, and let $\Phi$ be a homomorphism from $A$ to $B$. 
Assume that $\tau_{B}$ is faithful and $\tau_A=\tau_B\circ \Phi$. 
Then $F(\Phi(A), B)$ has a tracial state $\tau_{B, \omega}$ such that 
$\tau_{B,\omega}([(a_n)_n])=\lim_{n\to\omega}\tau_B(a_n)$ for any $[(a_n)_n]\in F(\Phi(A), B)$ by 
\cite[Proposition 2.1]{Na4}. 
Put 
$$
\mathcal{M}:= \ell^{\infty}(\mathbb{N}, \pi_{\tau_B}(B)^{''})/
\{\{x_n\}_{n\in\mathbb{N}}\in \ell^{\infty}(\mathbb{N}, \pi_{\tau_B}(B)^{''})\; |\; 
\lim_{n\to\omega}\tilde{\tau}_B(x_n^*x_n)=0\}
$$
where $\tilde{\tau}_B$ is the unique normal extension of $\tau_B$ on $\pi_{\tau_B}(B)^{''}$, and let 
$$
\mathcal{M}(\Phi(A), B):= \mathcal{M}\cap \pi_{\tau_B}(\Phi(A))^{\prime}.
$$
If $A=B$ and $\Phi=\mathrm{id}_A$, then we denote $\mathcal{M}(\Phi(A), B)$ by 
$\mathcal{M}(B)$. Note that $\mathcal{M}(B)$ is equal to the von Neumann algebraic central sequence 
algebra of $\pi_{\tau_{B}}(B)^{''}$. For any homomorphism $\Phi$ from $A$ to $B$, 
$\mathcal{M}(B)$ is a subalgebra of $\mathcal{M}(\Phi(A), B)$. 

Let $\beta$ be an action of a finite abelian group $\Gamma$ on $B$ such that 
$\beta_g(\Phi(A))=\Phi(A)$ for any $g\in\Gamma$.
Then $\beta$ induces an action $\tilde{\beta}$ on $\mathcal{M}(\Phi(A), B)$. 
By \cite[Proposition 3.11]{Na5}, we have the following proposition. Note that \cite[Proposition 3.11]{Na5} 
is based on Matui and Sato's techniques \cite{MS}, \cite{MS2} and \cite{MS3} (with pioneering works 
\cite{Sa0} and \cite{Sa}). 
See also \cite{Sa2} and \cite{Sza6}. 

\begin{pro}\label{pro:strict-comparison}
With notation as above, assume that $\mathcal{M}(\Phi(A), B)^{\tilde{\beta}}$ is a factor 
and $\beta_g|_{\Phi(A)}$ is outer for any $g\in\Gamma\setminus \{\iota\}$. If $a$ and $b$ 
are positive elements in 
$F(\Phi(A), B)^{\beta}$ satisfying $d_{\tau_{B},\omega}(a)< d_{\tau_{B}, \omega}(b)$, then there exists 
an element $r$ in $F(\Phi(A), B)^{\beta}$ such that $r^*br=a$. 
\end{pro}

\section{Approximate representability and classification}\label{sec:app}

We shall recall the definition of the Rohlin property and approximate representability for 
finite abelian group actions.

\begin{Def}
Let $A$ be a separable C$^*$-algebra, and let $\alpha$ be an action of a finite abelian group 
$\Gamma$ on $A$. \ \\
(i) We say that $\alpha$ has the \textit{Rohlin property} if there exists a partition of unity 
$\{p_g\}_{g\in\Gamma}$ consisting of projections in $F(A)$ such that 
$$
\alpha_g(p_h)=p_{gh}
$$ 
for any $g,h\in \Gamma$. \ \\
(ii) We say that $\alpha$ is \textit{approximately representable} if there exists a 
map $w$ from $\Gamma$ to $(A^{\alpha})^{\omega}$ such that 
the map $u$ from $\Gamma$ to $F(A^{\alpha})$ given by $u_g=[w_g]$ is a unitary representation of 
$\Gamma$ and 
$$
\alpha_g(a)=w_gaw_g^* \quad \text{in} \quad A^{\omega}
$$
for any $g\in\Gamma$ and $a\in A$. 
\end{Def}

We refer the reader to \cite{I1}, \cite{GSan1}, 
\cite{Na0} and \cite{San} for basic properties of the Rohlin property and 
approximate representability. See \cite{GHS} for some generalization. 

\begin{pro}\label{pro:rohlin-outer}
Let $A$ be a separable C$^*$-algebra, and let $\alpha$ be an action of a finite group 
$\Gamma$ on $A$. Assume that $\tau$ is an $\alpha$-invariant tracial state on $A$. 
If $\alpha$ has the Rohlin property, then $\tilde{\alpha}$ on 
$\pi_{\tau}(A)^{''}$ is outer. 
\end{pro}
\begin{proof}
Let $\mathcal{M}_{\omega}$ be a von Neumann algebraic central sequence algebra of 
$\pi_{\tau}(A)^{''}$. Then there exists a unital homomorphism from $F(A)$ to 
$\mathcal{M}_{\omega}$ (see, for example, \cite[Proposition 2.2]{Na2}). 
Hence there exists 
a partition of unity $\{P_g\}_{g\in\Gamma}$ consisting of projections in $\mathcal{M}_{\omega}$ 
such that 
$$
\tilde{\alpha}_g(P_h)=P_{gh}
$$ 
for any $g,h\in \Gamma$ since $\alpha$ has the Rohlin property. This shows that $\tilde{\alpha}$ 
is outer. Indeed, if $\tilde{\alpha}_g$ is an inner automorphism of 
$\pi_{\tau}(A)^{''}$, then $\tilde{\alpha}_g(P_{\iota})=P_{\iota}$. 
Therefore we have $P_{g}=P_{\iota}$, and hence $g=\iota$.  
\end{proof}

\begin{pro}\label{pro:unitary}
Let $A$ be a separable C$^*$-algebra, and let $\alpha$ be an action of a finite group 
$\Gamma$ on $A$. 
For any $v\in (A^{\alpha})^{\omega}\cap (A^{\alpha})^{\prime}$, if 
$[v]$ is a unitary element in $F(A^{\alpha})$, then $v^*va=vv^*a=av^*v=avv^*=a$ for any $a\in A 
\subset A^{\omega}$. 
\end{pro}
\begin{proof}
Let $\{h_n\}_{n=1}^{\infty}$ be an approximate unit for $A^{\alpha}$. 
Then $v^*vh_n=h_n$ in $A^{\omega}$ for any $n\in\mathbb{N}$ because $[v]$ is a unitary element in 
$F(A^{\alpha})$. 
Since $\Gamma$ is a finite group, $A^{\alpha}\subset A$ is a nondegenerate inclusion. 
Hence $\{h_n\}_{n=1}^{\infty}$ is an approximate unit for $A$. Therefore, for any $a\in A$,  
we have 
$$
\| v^*va-a\| =\|v^*va- v^*vh_na+ v^*vh_na-a\| \leq \|v^*v\| \|a-h_na\|+ \| h_na-a\|\to 0
$$
as $n\to \infty$. Consequently, $v^*va=a$. Similar arguments show 
$vv^*a=av^*v=avv^*=a$. 
\end{proof}

Let $A$ be a simple separable nuclear monotracial C$^*$-algebra, and let $\alpha$ be 
an outer action of a finite abelian group $\Gamma$ on $A$. We shall consider the action 
$\gamma:=\alpha\otimes\mathrm{id}_\mathcal{W}$ on $A\otimes\mathcal{W}$. 
We denote by $M_{n^{\infty}}$ the uniformly hyperfinite (UHF) algebra of type 
$n^{\infty}$. 
The following lemma is based on \cite[Lemma 3.10]{I1} and \cite[Proposition 2.1.3]{sut}. 

\begin{lem}\label{lem:sutherland}
Let $A$ be a separable C$^*$-algebra, and let $\alpha$ be 
an action of a finite abelian group $\Gamma$ on $A$ and put 
$\gamma=\alpha\otimes\mathrm{id}_{\mathcal{W}}$. Assume that for any $g\in\Gamma$, there 
exists an element $v_g$ in $((A\otimes\mathcal{W})^{\gamma})^{\omega}$ 
such that 
$$
\gamma_g(a)=v_gav_g^* \quad \text{in} \quad (A\otimes\mathcal{W})^{\omega}
$$
for any $a\in A\otimes\mathcal{W}$ and $[v_{g}]$ is a unitary element  in 
$F((A\otimes\mathcal{W})^{\gamma})$. 
Then $\gamma$ on 
$A\otimes\mathcal{W}$ is approximately representable. 
\end{lem}
\begin{proof}
Since $\mathcal{W}$ is isomorphic to $\mathcal{W}\otimes M_{|\Gamma|^{\infty}}$, there exists a 
unital homomorphism $\psi$ from $M_{|\Gamma|}(\mathbb{C})$ to 
$F(A\otimes\mathcal{W})^{\gamma}$. 
For any $g,h\in \Gamma$, let 
$E_{g,h}\in (A\otimes\mathcal{W})^{\omega}\cap (A\otimes\mathcal{W})^{\prime}$ be a 
representative of $\psi (e_{g,h})$ where $\{e_{g,h}\}_{g,h\in \Gamma}$ are the matrix units 
of $M_{|\Gamma|}(\mathbb{C})$. 
Taking suitable subsequences, we may assume that $E_{g,h}\in \{v_k, v_k^*\; |\; k\in\Gamma\}^{\prime}$ 
for any $g,h\in\Gamma$. Moreover, we may assume that 
$E_{g,h}\in ((A\otimes\mathcal{W})^{\gamma})^{\omega}$ for any 
$g,h\in \Gamma$ by replacing $E_{g,h}$ with $\frac{1}{|\Gamma|}\sum_{k\in \Gamma}
\gamma_g(E_{g,h})$. 
For any $g\in\Gamma$, let $z_g:=\sum_{h\in \Gamma}v_gv_hv_{gh}^*E_{h, gh}$.
Note that we have $v_ga=\gamma_g(a)v_g$ and 
$\gamma_{g^{-1}}(a)v_g^*=v_g^*a$ in $(A\otimes \mathcal{W})^{\omega}$ for any 
$a\in A\otimes\mathcal{W}$ and $g\in\Gamma$ by Proposition \ref{pro:unitary}. 
Hence we have 
$z_g\in  ((A\otimes\mathcal{W})^{\gamma})^{\omega}\cap
(A\otimes\mathcal{W})^{\prime}$ for any $g\in\Gamma$. 
Define a map $w$ from $\Gamma$ to 
$((A\otimes\mathcal{W})^{\gamma})^{\omega}$ by $w_g:= z_g^*v_g$ 
for any $g\in \Gamma$. 
Note that we have  
\begin{align*}
z_g^*z_ga
&= \sum_{h\in \Gamma}E_{gh,h}v_{gh}v_h^*v_g^*\sum_{k\in \Gamma}v_gv_kv_{gk}^*E_{k, gk}a 
= \sum_{h,k\in \Gamma}v_{gh}v_h^*v_g^*v_gv_kv_{gk}^*E_{gh, h}E_{k, gk}a \\
&= \sum_{h\in \Gamma}v_{gh}v_h^*v_g^*v_gv_hv_{gh}^*E_{gh, gh}a 
= \sum_{h\in \Gamma}v_{gh}v_h^*v_g^*v_g\gamma_{g^{-1}}(a)v_hv_{gh}^*E_{gh, gh} \\
&= \sum_{h\in \Gamma} v_{gh}v_h^*\gamma_{g^{-1}}(a)v_hv_{gh}^*E_{gh, gh} 
= \sum_{h\in \Gamma} v_{gh}v_h^*v_h\gamma_{h^{-1}g^{-1}}(a)v_{gh}^*E_{gh, gh} \\
&= \sum_{h\in \Gamma} v_{gh}\gamma_{h^{-1}g^{-1}}(a)v_{gh}^*E_{gh, gh} 
= \sum_{h\in \Gamma} v_{gh}v_{gh}^*aE_{gh, gh} = \sum_{h\in \Gamma} aE_{gh, gh}=a
\end{align*}
in $(A\otimes\mathcal{W})^{\omega}$ for any 
$a\in A\otimes\mathcal{W}$ and $g\in \Gamma$. 
Hence we have 
$$
w_gaw_g^*=z_g^*v_gav_g^*z_g=z_g^*\gamma_g(a)z_g
=z_g^*z_g\gamma_g(a)=\gamma_g(a)
$$
in $(A\otimes\mathcal{W})^{\omega}$ for any 
$a\in A\otimes\mathcal{W}$ and $g\in\Gamma$. 
We shall show the map $u$ from $\Gamma$ to $F((A\otimes\mathcal{W})^{\gamma})$ 
given by $u_g=[w_g]$ is a unitary representation. 
In a similar way as above, we see that  
$
z_gz_g^*a=a 
$ 
in $(A\otimes\mathcal{W})^{\omega}$ for any $a\in A\otimes\mathcal{W}$ and $g\in \Gamma$. 
Hence the image of $u$ is contained in the unitary group of $F((A\otimes\mathcal{W})^{\gamma})$. 
Note that we have 
\begin{align*}
v_gz_hv_g^*z_gz_{gh}^*a
&=v_gz_hv_g^*\sum_{k\in \Gamma}v_gv_kv_{gk}^*E_{k, gk}\sum_{k^{\prime}\in \Gamma}E_{ghk^{\prime},
k^{\prime}}
v_{ghk^{\prime}}v_{k^{\prime}}^*v_{gh}^*a \\
&= v_gz_hv_g^*\sum_{k,k^{\prime}\in \Gamma}v_gv_kv_{gk}^*v_{ghk^{\prime}}v_{k^{\prime}}^*v_{gh}^*
E_{k, gk}E_{ghk^{\prime},k^{\prime}}a \\
&= v_gz_h\sum_{k\in \Gamma}v_g^*v_gv_kv_{gk}^*v_{gk}v_{h^{-1}k}^*v_{gh}^*E_{k,h^{-1}k}a \\
&= v_gz_h\sum_{k\in \Gamma}v_kv_{h^{-1}k}^*v_{gh}^*E_{k,h^{-1}k}a \\
&= v_g\sum_{k, k^{\prime}\in \Gamma}v_hv_{k^{\prime}}v_{hk^{\prime}}^*
v_kv_{h^{-1}k}^*v_{gh}^*E_{k^{\prime}, hk^{\prime}}E_{k, h^{-1}k}a \\
&= v_g\sum_{k^{\prime}\in \Gamma}v_hv_{k^{\prime}}v_{hk^{\prime}}^*
v_{hk^{\prime}}v_{k^{\prime}}^*v_{gh}^*E_{k^{\prime}, k^{\prime}}a \\
&= v_{g}v_{h}v_{gh}^*\sum_{k^{\prime}\in \Gamma}E_{k^{\prime}, k^{\prime}}a
=v_{g}v_{h}v_{gh}^*a 
\end{align*}
in $(A\otimes\mathcal{W})^{\omega}$ for any $a\in A\otimes\mathcal{W}$ and $g, h\in \Gamma$.
This implies that 
$$
z_{gh}^*v_{gh}\gamma_{h^{-1}g^{-1}}(a) =z^*_gv_gz_{h}^*v_{h} \gamma_{h^{-1}g^{-1}}(a) 
\quad \text{in} \quad (A\otimes\mathcal{W})^{\omega}
$$
for any $a\in A\otimes\mathcal{W}$ and $g, h\in \Gamma$.
Consequently, we have $u_{gh}=u_{g}u_{h}$ for any $g,h\in \Gamma$. 
\end{proof}

We have the following lemma by \cite[Corollary 4.6]{Na5}. 

\begin{lem}\label{lem:corollary 4.6}
With notation as above, let  $p$ and $q$ be projections in 
$F(A\otimes\mathcal{W})^{\gamma}$ such that  
$0<\tau_{A\otimes\mathcal{W}, \omega}(p)\leq 1$.
Then $p$ and $q$ are  Murray-von Neumann equivalent if and only if 
$\tau_{A\otimes\mathcal{W}, \omega}(p)=\tau_{A\otimes\mathcal{W}, \omega}(q)$.  
\end{lem}

Fix $g_0\in \Gamma$. Define a homomorphism
$\Phi_{g_0}$ from $A\otimes\mathcal{W}$ to $M_2(A\otimes\mathcal{W})$ by 
$$
\Phi_{g_0}(a)= \left(\begin{array}{cc}
               a   &   0    \\ 
               0   &   \gamma_{g_0}(a)    
 \end{array} \right).
$$
Since $\Gamma$ is abelian, we have 
$\gamma_g
\otimes\mathrm{id}_{M_2(\mathbb{C})}(\Phi_{g_0}(A\otimes\mathcal{W}))=\Phi_{g_0}(A\otimes\mathcal{W})$ 
for any $g\in\Gamma$. 
Therefore $\gamma\otimes\mathrm{id}_{M_2(\mathbb{C})}$ 
induces an action on $F(\Phi_{g_0}(A\otimes\mathcal{W}), M_2(A\otimes\mathcal{W}))$. 
We denote it by $\beta$. Also, let $\tau_{\omega}$ denote the induced tracial state on 
$F(\Phi_{g_0}(A\otimes\mathcal{W}), M_2(A\otimes\mathcal{W}))$ by 
$\tau_{M_2(A\otimes\mathcal{W})}$. 

\begin{lem}\label{lem:strict-comparison}
With notation as above, assume that the characteristic invariant of $\tilde{\alpha}$ is trivial. 
If $a$ and $b$ 
are positive elements in 
$F(\Phi_{g_0}(A\otimes\mathcal{W}), M_2(A\otimes\mathcal{W}))^{\beta}$ satisfying 
$d_{\tau_{\omega}}(a)< d_{\tau_{\omega}}(b)$, 
then there exists an element $r$ in $F(\Phi_{g_0}(A\otimes\mathcal{W}), 
M_2(A\otimes\mathcal{W}))^{\beta}$ such that $r^*br=a$. 
\end{lem}
\begin{proof}
By Proposition \ref{pro:strict-comparison}, it suffices to show that 
$\mathcal{M}(\Phi_{g_0}(A\otimes\mathcal{W}), M_2(A\otimes\mathcal{W}))^{\tilde{\beta}}$ is 
a factor. Since the characteristic invariant of $\tilde{\alpha}$ is trivial, there exists a group 
homomorphism $v$ from $N(\tilde{\alpha})$ to the unitary group of 
$(\pi_{\tau_{A}}(A)^{''})^{\tilde{\alpha}}$ such that $\tilde{\alpha}_g=\mathrm{Ad}(v_g)$ for any 
$g\in N(\tilde{\alpha})$. Since we have 
$$
\tilde{\gamma}_g\otimes\mathrm{id}_{M_2(\mathbb{C})}= \mathrm{Ad}\left( \left(\begin{array}{cc}
              v_g\otimes  1_{\pi_{\tau_{\mathcal{W}}}(\mathcal{W})^{''}}  &   0    \\ 
               0    &      v_g\otimes 1_{\pi_{\tau_{\mathcal{W}}}(\mathcal{W})^{''}} 
 \end{array} \right)\right)
$$
and $\tilde{\beta}_g([(x_n)_n])=[(\tilde{\gamma}_g\otimes\mathrm{id}_{M_2(\mathbb{C})}(x_n))_n]$, 
$$
\tilde{\beta_g}([(x_n)_n])= \left[\left(\mathrm{Ad}\left( \left(\begin{array}{cc}
              v_g\otimes  1_{\pi_{\tau_{\mathcal{W}}}(\mathcal{W})^{''}}  &   0    \\ 
               0    &      v_g\otimes 1_{\pi_{\tau_{\mathcal{W}}}(\mathcal{W})^{''}} 
 \end{array} \right)\right)(x_n)\right)_n\right]
$$
for any for any $g\in N_{\tilde{\alpha}}$ and $[(x_n)_n]$ in 
$\mathcal{M}(\Phi_{g_0}(A\otimes\mathcal{W}), M_2(A\otimes\mathcal{W}))^{\tilde{\beta}}$. 
Note that  
$$
\left(\begin{array}{cc}
              v_g\otimes  1_{\pi_{\tau_{\mathcal{W}}}(\mathcal{W})^{''}}  &   0    \\ 
               0    &      v_g\otimes 1_{\pi_{\tau_{\mathcal{W}}}(\mathcal{W})^{''}} 
 \end{array} \right)
\in 
\pi_{\tau_{\omega}}(\Phi_{g_0}(A\otimes\mathcal{W}))^{''}
$$
because we have $v_g\otimes  1_{\pi_{\tau_{\mathcal{W}}}(\mathcal{W})^{''}}\in 
\pi_{\tau_{A\otimes\mathcal{W}}}(A\otimes\mathcal{W})^{''}$ and 
$$
\left(\begin{array}{cc}
              v_g\otimes  1_{\pi_{\tau_{\mathcal{W}}}(\mathcal{W})^{''}}  &   0    \\ 
               0    &      v_g\otimes 1_{\pi_{\tau_{\mathcal{W}}}(\mathcal{W})^{''}} 
 \end{array} \right)
=\left(\begin{array}{cc}
              v_g\otimes  1_{\pi_{\tau_{\mathcal{W}}}(\mathcal{W})^{''}}  &   0    \\ 
               0    &      \tilde{\gamma}_{g_0}(v_g\otimes 1_{\pi_{\tau_{\mathcal{W}}}(\mathcal{W})^{''}}) 
 \end{array} \right).
$$
Hence $\tilde{\beta}_g$ is the trivial automorphism for any $g\in N(\tilde{\alpha})$. 
Therefore $\gamma\otimes\mathrm{id}_{M_2(\mathbb{C})}$ induces 
an action $\delta$ of $\Gamma /N(\tilde{\alpha})$ on 
$\mathcal{M}(\Phi_{g_0}(A\otimes\mathcal{W}), M_2(A\otimes\mathcal{W}))$ such that 
$$
\mathcal{M}(\Phi_{g_0}(A\otimes\mathcal{W}), M_2(A\otimes\mathcal{W}))^{\delta}
=
\mathcal{M}(\Phi_{g_0}(A\otimes\mathcal{W}), M_2(A\otimes\mathcal{W}))^{\tilde{\beta}}.
$$
Note that the restriction of $\delta$ on $\mathcal{M}(M_2(A\otimes\mathcal{W}))$ is strongly free or 
$\Gamma /N(\tilde{\alpha})=\{\iota\}$ 
by \cite[Theorem 3.2]{C3}, \cite[Lemma 5.6]{Oc} and the definition of $N(\tilde{\alpha})$. 
(Note that every centrally nontrivial automorphism of a factor is properly centrally nontrivial by 
definition. See \cite[Section 5.2]{Oc}.) 
Therefore \cite[Proposition 3.14]{Na5}(see also \cite[Remark 3.15]{Na5}) 
implies that 
$
\mathcal{M}(\Phi_{g_0}(A\otimes\mathcal{W}), M_2(A\otimes\mathcal{W}))^{\delta}
$
is a factor. Consequently, $\mathcal{M}(\Phi_{g_0}(A\otimes\mathcal{W}),
M_2(A\otimes\mathcal{W}))^{\tilde{\beta}}$ is 
a factor.
\end{proof}

The following theorem is one of the main results in this paper. 

\begin{thm}\label{thm:main}
Let $A$ be a simple separable nuclear monotracial C$^*$-algebra, and let $\alpha$ be 
an outer action of a finite abelian group $\Gamma$ on $A$. 
Then $\gamma=\alpha\otimes\mathrm{id}_{\mathcal{W}}$ on $A\otimes\mathcal{W}$ is approximately 
representable if and only if the characteristic invariant of $\tilde{\alpha}$ is trivial. 
\end{thm}
\begin{proof}
First, we shall show the only if part. Assume that the characteristic invariant of $\tilde{\alpha}$ is 
not trivial. By Proposition \ref{pro:non-trivial-characteristic}, the dual action of 
$\tilde{\alpha}$ on $\pi_{\tau_{A}}(A)^{''}\rtimes_{\tilde{\alpha}}\Gamma$ 
is not outer, and 
hence the dual action of $\tilde{\gamma}$ on 
$\pi_{\tau_{A}\otimes\tau_{\mathcal{W}}}(A\otimes\mathcal{W})^{''}\rtimes_{\tilde{\gamma}}\Gamma
\cong\pi_{\tau_A\otimes\tau_{\mathcal{W}}\circ E_{\gamma}}((A\otimes\mathcal{W})\rtimes_{\gamma}\Gamma)^{''}$ 
is not outer. 
Proposition \ref{pro:rohlin-outer} implies that $\hat{\gamma}$ does not have the Rohlin property. 
Therefore $\gamma$ is not approximately representable by \cite[Proposition 4.4]{Na0}. 

We shall show the if part. 
Fix $g_0\in \Gamma$. Let $\{h_n\}_{n\in\mathbb{N}}$ be an approximate unit 
for $(A\otimes\mathcal{W})^{\gamma}$. Note that $\{h_n\}_{n\in\mathbb{N}}$ is also an approximate 
unit for $A\otimes\mathcal{W}$. 
Put 
$$
P:=\left[\left(\left(\begin{array}{cc}
              h_n  &   0    \\ 
               0    &   0 
 \end{array} \right)\right)_n\right]
\quad \text{and} 
\quad
Q:=\left[\left(\left(\begin{array}{cc}
               0  &   0    \\ 
               0  &   h_n 
 \end{array} \right)\right)_n\right]
$$
in 
$F(\Phi_{g_0}(A\otimes\mathcal{W}), M_2(A\otimes\mathcal{W}))$. 
Then $P$ and $Q$ are projections in 
$F(\Phi_{g_0}(A\otimes\mathcal{W}), M_2(A\otimes\mathcal{W}))^{\beta}$. 
Using \cite[Proposition 4.2]{Na5}, Lemma \ref{lem:corollary 4.6} and Lemma \ref{lem:strict-comparison}
instead of \cite[Proposition 2.6]{Na3}, \cite[Corollary 5.5]{Na3} and \cite[Lemma 6.1]{Na3}, 
similar arguments as in the proof of \cite[Lemma 6.2]{Na3} show that $P$ is Murray-von Neumann 
equivalent to $Q$ in $F(\Phi_{g_0}(A\otimes\mathcal{W}), M_2(A\otimes\mathcal{W}))^{\beta}$. 
(See also the proof of \cite[Lemma 4.2]{Na4}.)
Hence there exists an element $V$ in 
$F(\Phi_{g_0}(A\otimes\mathcal{W}), M_2(A\otimes\mathcal{W}))^{\beta}$ such that $V^*V=P$ and 
$VV^*=Q$. 
It is easy to see that there exists an element $v_{g_0}=(v_{g_0, n})_n$ 
in $(A\otimes\mathcal{W})^{\omega}$ 
such that 
$$
V=\left[\left(\left(\begin{array}{cc}
               0  &   0    \\ 
             v_{g_0,n}  &   0 
 \end{array} \right)\right)_n\right].
$$ 
Since we have $\beta_g(V)=V$ for any $g\in\Gamma$, we see that 
$$
\left[\left(\left(\begin{array}{cc}
               0  &   0    \\ 
             v_{g_0,n}  &   0 
 \end{array} \right)\right)_n\right]
=
\left[\left(\left(\begin{array}{cc}
               0  &   0    \\ 
             \frac{1}{|\Gamma|}\sum_{g\in\Gamma}\gamma_g(v_{g_0,n})  &   0 
 \end{array} \right)\right)_n\right].
$$
Hence we may assume that $v_{g_0}$ is an element in $((A\otimes\mathcal{W})^{\gamma})^{\omega}$. 
Since we have $V^*V=P$ and $VV^*=Q$, 
$$
av_{g_0}^*v_{g_0}=av_{g_0}v_{g_0}^*=a
$$ 
for any $a\in A\otimes\mathcal{W}$. 
Furthermore, 
we have 
$$
v_{g_0}a=\gamma_{g_0}(a)v_{g_0}
$$ 
for any $a\in A\otimes\mathcal{W}$ since 
$$
\left(\left(\begin{array}{cc}
               0      &   0    \\ 
          v_{g_0,n}  &   0 
 \end{array} \right)\right)_n
\in M_2(A\otimes\mathcal{W})^{\omega}\cap \Phi_{g_0}(A\otimes\mathcal{W})^{\prime}.
$$ 
These imply that 
$$
\gamma_{g_0}(a)=v_{g_0}av_{g_0}^*
$$
for any $a\in A\otimes\mathcal{W}$ 
and  
$[v_{g_0}]$ is a unitary element in $F((A\otimes\mathcal{W})^{\gamma})$. 
Since $g_0\in\Gamma$ is arbitrary, $\gamma$ is approximately representable 
by Lemma \ref{lem:sutherland}.
\end{proof}

Since $(A\otimes\mathcal{W})\rtimes_{\alpha\otimes\mathrm{id}_{\mathcal{W}}}\Gamma$ 
is isomorphic to $(A\rtimes_{\alpha}\Gamma)\otimes \mathcal{W}$, we see that 
$(A\otimes\mathcal{W})\rtimes_{\alpha\otimes\mathrm{id}_{\mathcal{W}}}\Gamma$ is in 
the class of Elliott-Gong-Lin-Niu's classification theorem \cite[Theorem 7.5]{EGLN} 
(see also \cite[Theorem A]{CE}). 
Furthermore, we see that 
$(A\otimes\mathcal{W})\rtimes_{\alpha\otimes\mathrm{id}_{\mathcal{W}}}\Gamma$ 
is in the class of Robert's classification theorem \cite{Rob}. 
In particular, these C$^*$-algebras and automorphisms can 
be classified by using trace spaces. Note that the map
from $T_1(A\rtimes_{\alpha}\Gamma)$ to $T_1( (A\rtimes_{\alpha}\Gamma)\otimes \mathcal{W})$ 
given by $\tau \mapsto \tau\otimes \tau_{\mathcal{W}}$ is an affine homeomorphism. 
As an application of the theorem above and these classification theorem, 
we obtain the following classification result. 

\begin{thm}
Let $A$ and $B$ be 
simple separable nuclear monotracial C$^*$-algebras, and let $\alpha$ and $\beta$ be 
outer actions of a finite abelian group $\Gamma$ on $A$ and $B$, respectively. 
Assume that the characteristic invariants of $\tilde{\alpha}$ and $\tilde{\beta}$ are trivial. 
Then \ \\
(i) $\alpha\otimes \mathrm{id}_{\mathcal{W}}$ on $A\otimes\mathcal{W}$ and 
$\beta\otimes \mathrm{id}_{\mathcal{W}}$ on $B\otimes\mathcal{W}$ are cocycle conjugate 
if and only if 
$\tilde{\alpha}$ on $\pi_{\tau_A}(A)^{''}$ and $\tilde{\beta}$ on $\pi_{\tau_B}(B)^{''}$ are 
cocycle conjugate; \ \\
(ii) $\alpha\otimes \mathrm{id}_{\mathcal{W}}$ on $A\otimes\mathcal{W}$ and 
$\beta\otimes \mathrm{id}_{\mathcal{W}}$ on $B\otimes\mathcal{W}$ are conjugate if and only if 
$\tilde{\alpha}$ on $\pi_{\tau_A}(A)^{''}$ and $\tilde{\beta}$ on $\pi_{\tau_B}(B)^{''}$ are conjugate.
\end{thm}
\begin{proof}
(i) First, we shall show the only if part. Since $\alpha\otimes \mathrm{id}_{\mathcal{W}}$ and 
$\beta\otimes \mathrm{id}_{\mathcal{W}}$ are cocycle conjugate, 
$\tilde{\alpha}\otimes\mathrm{id}_{\pi_{\tau_{\mathcal{W}}}(\mathcal{W})^{''}}$ and 
$\tilde{\beta}\otimes\mathrm{id}_{\pi_{\tau_{\mathcal{W}}}(\mathcal{W})^{''}}$ are cocycle conjugate. 
Since $\tilde{\alpha}\otimes\mathrm{id}_{\pi_{\tau_{\mathcal{W}}}(\mathcal{W})^{''}}$ is conjugate to 
$\tilde{\alpha}$ by \cite[Corollary 5.2.3]{Jones}, 
we see that $\tilde{\alpha}$ and $\tilde{\beta}$ are cocycle conjugate. 

We shall show the if part. 
Since $\tilde{\alpha}$ and $\tilde{\beta}$ are cocycle conjugate, 
 there exists an affine homeomorphism $F$
from $T_1(\pi_{\tau_B}(B)^{''}\rtimes_{\tilde{\beta}}\Gamma)$ onto 
$T_1(\pi_{\tau_A}(A)^{''}\rtimes_{\tilde{\alpha}}\Gamma)$ such that 
$F\circ T(\hat{\tilde{\beta}}_{\eta})=T(\hat{\tilde{\alpha}}_{\eta})\circ F$ for any 
$\eta\in \hat{\Gamma}$ by Proposition \ref{pro:conjugacy-trace-spaces}. 
Proposition \ref{pro:trace-spaces-crossed-products} implies that the restriction map 
$F|_{T_1(B\rtimes_{\beta}\Gamma)}$ is an affine homeomorphism from 
$T_1(B\rtimes_{\beta}\Gamma)$ onto $T_1(A\rtimes_{\alpha}\Gamma)$. 
Define a map $G$ from 
$T_1((B\otimes\mathcal{W})\rtimes_{\beta\otimes\mathrm{id}_{\mathcal{W}}}\Gamma)$ to 
$T_1((A\otimes\mathcal{W})\rtimes_{\alpha\otimes\mathrm{id}_{\mathcal{W}}}\Gamma)$ 
by $G(\tau\otimes\tau_{\mathcal{W}})= F(\tau)\otimes \tau_{\mathcal{W}}$ for any 
$\tau\in T_1(B\rtimes_{\beta}\Gamma)$. Then $G$ is an affine homeomorphism 
such that $G\circ T(\hat{\beta}_\eta\otimes\mathrm{id}_{\mathcal{W}})
= T(\hat{\alpha}_{\eta}\otimes\mathrm{id}_{\mathcal{W}})\circ G$ for any $\eta\in\hat{\Gamma}$. 
By Elliott-Gong-Lin-Niu's classification theorem \cite[Theorem 7.5]{EGLN}, there exists an isomorphism 
$\theta$ from $(A\otimes\mathcal{W})\rtimes_{\alpha\otimes\mathrm{id}_{\mathcal{W}}}\Gamma$ onto 
$(B\otimes\mathcal{W})\rtimes_{\beta\otimes\mathrm{id}_{\mathcal{W}}}\Gamma$ 
such that $T(\theta)=G$. 
Since we have 
$$
T(\theta \circ \hat{\alpha}_{\eta}\otimes\mathrm{id}_{\mathcal{W}} \circ \theta^{-1})
=G^{-1}\circ T(\hat{\alpha}_{\eta}\otimes\mathrm{id}_{\mathcal{W}})\circ G
=T(\hat{\beta}_{\eta}\otimes\mathrm{id}_{\mathcal{W}}),
$$ 
$\hat{\beta}_{\eta}\otimes\mathrm{id}_{\mathcal{W}}$ is 
approximately unitarily equivalent to 
$\theta \circ \hat{\alpha}_{\eta}\otimes\mathrm{id}_{\mathcal{W}} \circ \theta^{-1}$ for any 
$\eta\in \hat{\Gamma}$ by \cite[Theorem 1.0.1]{Rob} and \cite[Proposition 6.2.3]{Rob}.
Therefore \cite[Theorem 3.5]{Na0} implies that $\hat{\alpha}\otimes\mathrm{id}_{\mathcal{W}}$ and 
$\hat{\beta}\otimes\mathrm{id}_{\mathcal{W}}$ are conjugate because 
$\hat{\alpha}\otimes\mathrm{id}_{\mathcal{W}}$ and 
$\hat{\beta}\otimes\mathrm{id}_{\mathcal{W}}$ have the Rohlin property
by Theorem \ref{thm:main} and \cite[Proposition 4.4]{Na0}.  
Consequently, \cite[Proposition 5.4]{Na0} implies that $\alpha\otimes\mathrm{id}_{\mathcal{W}}$ 
and $\beta\otimes\mathrm{id}_{\mathcal{W}}$ are cocycle conjugate. 
\ \\
(ii) Since we can show the only if part by the same argument as in (i), we shall show the if part. 
By Proposition \ref{pro:conjugacy-trace-spaces}, there exists 
an affine homeomorphism $F$
from $T_1(\pi_{\tau_B}(B)^{''}\rtimes_{\tilde{\beta}}\Gamma)$ onto 
$T_1(\pi_{\tau_B}(B)^{''}\rtimes_{\tilde{\alpha}}\Gamma)$ such that 
$F(\tau)(e_{\tilde{\alpha}})=\tau (e_{\tilde{\beta}})$ for any $\tau\in 
T_1(\pi_{\tau_B}(B)^{''}\rtimes_{\tilde{\alpha}}\Gamma)$ and 
$F\circ T(\hat{\tilde{\beta}}_{\eta})=T(\hat{\tilde{\alpha}}_{\eta})\circ F$ for any 
$\eta\in \hat{\Gamma}$. 
Note that we have $e_{\alpha}=e_{\tilde{\alpha}}$ and $e_{\beta}=e_{\tilde{\beta}}$ 
because we regard $M(A\rtimes_{\alpha}\Gamma)$ and $M(B\rtimes_{\beta}\Gamma)$ as 
subalgebras of  $\pi_{\tau_A}(A)^{''}\rtimes_{\tilde{\alpha}}\Gamma$ and 
$\pi_{\tau_B}(B)^{''}\rtimes_{\tilde{\beta}}\Gamma$, respectively. 
By the same argument as in (i), we see that there exists 
an isomorphism 
$\theta$ from $(A\otimes\mathcal{W})\rtimes_{\alpha\otimes\mathrm{id}_{\mathcal{W}}}\Gamma$ onto 
$(B\otimes\mathcal{W})\rtimes_{\beta\otimes\mathrm{id}_{\mathcal{W}}}\Gamma$ 
such that 
$\hat{\beta}_{\eta}\otimes\mathrm{id}_{\mathcal{W}}$ is 
approximately unitarily equivalent to 
$\theta\circ \hat{\alpha}_{\eta}\otimes\mathrm{id}_{\mathcal{W}} \circ \theta^{-1}$ for any 
$\eta\in \hat{\Gamma}$. 
Since we have  
$e_{\alpha\otimes\mathrm{id}_{\mathcal{W}}}= e_{\alpha}\otimes 1_{\mathcal{W}^{\sim}}$ and 
$e_{\beta\otimes\mathrm{id}_{\mathcal{W}}}= e_{\beta}\otimes 1_{\mathcal{W}^{\sim}}$, 
$$
\tau\otimes\tau_{\mathcal{W}} (\theta (e_{\alpha\otimes\mathrm{id}_{\mathcal{W}}}))
=F(\tau)\otimes \tau_{\mathcal{W}}(e_{\alpha}\otimes1_{\mathcal{W}})
=\tau \otimes\tau_{\mathcal{W}}(e_{\beta}\otimes1_{\mathcal{W}})
=\tau \otimes\tau_{\mathcal{W}}(e_{\beta\otimes\mathrm{id}_{\mathcal{W}}})
$$
for any $\tau\in T_1(B\rtimes_{\beta}\Gamma)$. Therefore  
\cite[Corollary 4.5]{Na0} implies that $\alpha\otimes\mathrm{id}_{\mathcal{W}}$ 
and $\beta\otimes\mathrm{id}_{\mathcal{W}}$ are conjugate because 
$\alpha\otimes\mathrm{id}_{\mathcal{W}}$ and $\beta\otimes\mathrm{id}_{\mathcal{W}}$ are 
approximately representable by Theorem \ref{thm:main}. 
\end{proof}

The following corollary is an immediate consequence of the theorem above and Theorem 
\ref{thm:jones}. 

\begin{cor}\label{main:cor}
Let $A$ and $B$ be 
simple separable nuclear monotracial C$^*$-algebras, and let $\alpha$ and $\beta$ be 
outer actions of a finite abelian group $\Gamma$ on $A$ and $B$, respectively. 
Assume that the characteristic invariants of $\tilde{\alpha}$ and $\tilde{\beta}$ are trivial. 
Then \ \\
(i) $\alpha\otimes \mathrm{id}_{\mathcal{W}}$ on $A\otimes\mathcal{W}$ and 
$\beta\otimes \mathrm{id}_{\mathcal{W}}$ on $B\otimes\mathcal{W}$ are cocycle conjugate 
if and only if $N(\tilde{\alpha})=N(\tilde{\beta})$; \ \\
(ii) $\alpha\otimes \mathrm{id}_{\mathcal{W}}$ on $A\otimes\mathcal{W}$ and 
$\beta\otimes \mathrm{id}_{\mathcal{W}}$ on $B\otimes\mathcal{W}$ are conjugate if and only if 
$N(\tilde{\alpha})=N(\tilde{\beta})$ and $i(\tilde{\alpha})=i(\tilde{\beta})$.
\end{cor}

If $\delta$ is an action of a finite cyclic group $\Gamma$ with prime order, 
then $N(\delta)=\Gamma$ or $N(\delta)=\{\iota\}$. 
Hence we have the following corollary. 

\begin{cor}
Let $A$ and $B$ be 
simple separable nuclear monotracial C$^*$-algebras, and let $\alpha$ and $\beta$ be 
outer actions of a finite cyclic group $\Gamma$ with prime order on $A$ and $B$, respectively. 
Then \ \\
(i) $\alpha\otimes \mathrm{id}_{\mathcal{W}}$ on $A\otimes\mathcal{W}$ and 
$\beta\otimes \mathrm{id}_{\mathcal{W}}$ on $B\otimes\mathcal{W}$ are cocycle conjugate 
if and only if $N(\tilde{\alpha})=N(\tilde{\beta})$; \ \\
(ii) $\alpha\otimes \mathrm{id}_{\mathcal{W}}$ on $A\otimes\mathcal{W}$ and 
$\beta\otimes \mathrm{id}_{\mathcal{W}}$ on $B\otimes\mathcal{W}$ are conjugate if and only if 
$N(\tilde{\alpha})=N(\tilde{\beta})$ and $i(\tilde{\alpha})=i(\tilde{\beta})$.
\end{cor}

\section{Model actions}

In this section, we shall construct simple separable nuclear monotracial C$^*$-algebras $A$ 
and outer actions $\alpha$ on $A$ with arbitrary invariants $(N(\tilde{\alpha}), i(\tilde{\alpha}))$ 
in Corollary \ref{main:cor} with the restrictions demanded by Proposition \ref{pro:realized-invariant}.
In particular, $A$ can be chosen to be approximately finite-dimensional (AF) algebras.

We shall give a summary of the construction.
First, we construct ``inner'' actions $\alpha$ on a simple monotracial AF algebra $A$ with arbitrary 
invariants in Lemma \ref{lem:jones}. 
Note that if $\beta$ is an action of $\Gamma$ on a simple monotracial C$^*$-algebra $B$ 
with $N(\beta)=\{\iota\}$ and $N(\tilde{\beta})=\Gamma$, then we have 
$N(\alpha \otimes \beta)=\{\iota\}$ and $N(\tilde{\alpha}\otimes\tilde{\beta})=N(\tilde{\alpha})$ 
for any  action of $\Gamma$ on a simple monotracial C$^*$-algebra $A$. 
Hence if we can construct a simple monotracial AF algebra $B$ and an action $\beta$ of 
$\Gamma$ on $B$ with $N(\beta)=\{\iota\}$ and $N(\tilde{\beta})=\Gamma$, then we obtain 
actions $\alpha\otimes\beta$ with arbitrary invariant $N(\tilde{\alpha}\otimes\tilde{\beta})$. 
Note that constructing such an action $\beta$ 
is equivalent to constructing a unitary representation $u$ of $\Gamma$ on 
$\pi_{\tau_B}(B)^{''}$ such that $u$ induces an action on $B$. We construct such representations in 
Lemma \ref{lem:outer}. (We construct unitary representations 
of cyclic groups for simplicity.) 
On the other hand, the inner invariant $i(\tilde{\alpha}\otimes\mathrm{Ad}(u))$ is not equal to 
$i(\tilde{\alpha})$ unless $|\tilde{\tau}_{B}(u_g)|=1$ for any $g\in \Gamma$. 
(Note that $|\tilde{\tau}_{B}(u_g)|=1$ for any $g\in \Gamma$ is equivalent to $\mathrm{Ad}(u_g)=
\mathrm{id}_{\pi_{\tau_B}(B)^{''}}$ for any $g\in \Gamma$.) 
But if $\tilde{\tau}_{B}(u_g)$ is ``near'' $1$ for any $g\in \Gamma$, then 
$i(\tilde{\alpha}\otimes\mathrm{Ad}(u))$ is ``near'' $i(\tilde{\alpha})$. 
Hence applying Lemma \ref{lem:jones} to a ``small'' perturbation of the desired inner 
invariant (we need to assume that $m$ has full support here) and considering the tensor 
product type action, we obtain the conclusion. 
Of course, we need to construct unitary representations $u$ in Lemma \ref{lem:outer} such that 
$\tilde{\tau}_{B}(u_g)$ is ``near'' $1$ for any $g\in \Gamma$. 

The following lemma is based on \cite[Proposition 1.5.8 and Theorem 1.5.11]{Jones}. 
Recall that $\Phi_{v}$ is the homomorphism from $\mathbb{C}N$ to $A$ defined by 
$\Phi_v(\sum_{h\in N}c_hh)=\sum_{h\in N}c_hv_h$ where $v$ is a unitary representation of $N$ 
on $A$. 

\begin{lem}\label{lem:jones}
Let $\Gamma$ be a finite abelian group, and let $N$ be a subgroup of $\Gamma$. 
Suppose that $m$ is a probability measure on the set $P_{N}$ of minimal projections in 
$\mathbb{C}N$. 
Then 
there exist a simple unital monotracial AF algebra $A$, an action $\alpha$ on $A$ of $\Gamma$ 
and a unitary representation $v$ of $N$ on $A$ such that $N(\alpha)=N(\tilde{\alpha})=N$, 
$\alpha_h=\mathrm{Ad}(v_h)$, $\alpha_g(v_h)=v_h$ for any $h\in N$ and $g\in \Gamma$ 
and $\tau_{A}(\Phi_{v}(p))=m(p)$ for any $p\in P_N$. 
\end{lem}
\begin{proof}
Define an action $\mu^{\Gamma}$ of $\Gamma$ on $M_{|\Gamma|^{\infty}}
\cong \bigotimes _{n=1}^{\infty} M_{|\Gamma|}(\mathbb{C})$ by 
$\mu^{\Gamma}:= \bigotimes_{n=1}^{\infty} \mathrm{Ad}(\lambda)$ where $\lambda$ is the left 
regular representation of $\Gamma$. 
Then $\mu^{\Gamma}$ is the Rohlin action (see, for example, \cite[Example 3.2]{I1} 
and \cite[1.5]{Jones}). Let $B:=M_{|\Gamma|^{\infty}}\rtimes_{\mu^{\Gamma}|_N}N$.
Then $B$ is a unital AF algebra because $\mu^{\Gamma}|_N$ is an action of product type. 
Since $N(\tilde{\mu}^{\Gamma}|_N)=\{\iota\}$, $B$ is simple and monotracial. 
Note that the unique tracial state on $B$ is given by 
$\tau_{M_{|\Gamma|^{\infty}}}\circ E_{\mu^{\Gamma}|_N}$. 
Define an action $\beta$ on 
$B$ of $\Gamma$ by $\beta_g(\sum_{h\in N}a_h\lambda_h)=
\sum_{h\in N}\mu^{\Gamma}_g(a_h)\lambda_h$ for any $g\in \Gamma$. 
By the same argument as in the proof 
of \cite[Proposition 1.5.8]{Jones}, we see that $N(\tilde{\beta})=N(\beta)=N$. 
In particular, the map $\lambda$ given by $N \ni h \mapsto \lambda_h\in B$ 
is a unitary representation of $N$ on $B$ 
such that $\beta_g(\lambda_h)=\lambda_h$ for any $g\in\Gamma$ and $h\in N$. 

The Effros-Handelman-Shen theorem \cite{EHS} (or \cite[Theorem 2.2]{Ell-order}) implies that 
there exists a simple unital  monotracial AF algebra $C$ such that 
$K_0(C)$ is the additive subgroup of $\mathbb{R}$ generated by $\mathbb{Q}$ and 
$\{m(p)\; |\; p\in P_{N}\}$, $K_0(C)_{+}=K_0(C)\cap \mathbb{R}_{+}$ and $[1]_0=1$. 
For any $p\in P_N$, there exists a projection $q_p$ in $C$ such that $\tau_{C}(q_p)=m(p)$, 
and put $e_{p}=\Phi_{\lambda}(p)\otimes q_p\in B\otimes C$. 
Then we have $e_{p}\in (B\otimes C)^{\beta\otimes\mathrm{id}_{C}}
\cap \{\lambda_h\otimes 1\; |\; h\in  N\}^{\prime}$ for any $p\in P_N$. 
Since there exists an element $\eta$ in $\hat{N}$ such that 
$\Phi_{\lambda}(p)=\frac{1}{|N|}\sum_{h\in N}\eta (h) \lambda_h$, 
$
\tau_{B} (\Phi_{\lambda}(p))= \frac{1}{|N|}
$
for any $p\in P_N$.  Hence we have $\tau_{B\otimes C}(e_p)= \frac{m(p)}{|N|}$ for any $p\in P_N$.

Put $e:= \sum_{p\in P_N}e_p$, and let $A:= e(B\otimes C)e$. 
Then $A$ is a simple unital monotracial 
AF algebra. Since we have $e\in  (B\otimes C)^{\beta\otimes\mathrm{id}_{C}}$, 
$\beta\otimes\mathrm{id}_{C}$ induces an action $\alpha$ on $A$ of $\Gamma$.
By the same reason as in the proof of \cite[Theorem 1.5.11]{Jones}, we have 
$N(\tilde{\alpha})=N(\tilde{\beta}\otimes \mathrm{id}_{C})=N$. Define a unitary representation $v$ of 
$N$ on $A$ by $v_h:= (\lambda_h\otimes 1)e$ for any $h\in N$. It is easy to see that 
$\alpha_h=\mathrm{Ad}(v_h)$ and $\alpha_g(v_h)=v_h$ for any $h\in N$ and $g\in \Gamma$. 
This also implies that $N(\alpha)=N$. Since we have $\Phi_{v}(p)=(\Phi_{\lambda}(p)\otimes 1)e
=e_{p}$, 
$$
\tau_{A}(\Phi_{v}(p))=\frac{\tau_{B\otimes C}(e_p)}{\tau_{B\otimes C}(e)}
=\frac{\frac{m(p)}{|N|}}{\frac{1}{|N|}}=m(p)
$$
for any $p\in P_N$. Therefore the proof is complete. 
\end{proof}

We recall properties of characters of finite abelian groups. If $N$ is 
a finite abelian group, then we have 
$$
\sum_{h\in N}\eta (h) 
=
\left\{\begin{array}{cl}
|N| & \text{if}\quad \eta=\iota \in \hat{N} \\
0 & \text{if}\quad \eta\in \hat{N}\setminus \{\iota\}
\end{array}
\right.
\quad\text{and}\quad
\sum_{\eta \in \hat{N}}\eta (h) 
=
\left\{\begin{array}{cl}
|N| & \text{if}\quad h=\iota \in N \\
0 & \text{if}\quad h\in N\setminus \{\iota \}.
\end{array}
\right. 
$$
We denote by $\mathbb{Z}_k$ the cyclic group of order $k$. For any natural number $k$, let 
$\zeta_{k}:=e^{\frac{2\pi i}{k}}$. Note that $\hat{\mathbb{Z}}_k$ can be identified with 
$\{\zeta_k^l\; | \; 1 \leq l \leq k\}$ by the pairing $\mathbb{Z}_k \times \hat{\mathbb{Z}_k}\to 
\mathbb{T}$ given by $([l], \zeta_k)\mapsto \zeta_k^{l}$.  
Also, if $\zeta$ is a root of unity and not equal to $1$, then we have $\sum_{j=1}^{n}\zeta^j=0$.

\begin{lem}\label{lem:outer}
Let $k$ be a natural number with $k\geq 2$ and $r$ a real number with $0<r<1$. 
Then there exist a simple unital monotracial  
AF algebra $A$ and a unitary element $V$ in 
$\pi_{\tau_{A}}(A)^{''}$ such that $\mathrm{Ad}(V)$ induces an 
outer action of $\mathbb{Z}_k$ on $A$, $V^{k}=1$ and 
$\tilde{\tau}_{A}(V^{l})=r$ for any $1\leq l \leq k-1$.  
\end{lem}
\begin{proof}
By the Effros-Handelman-Shen theorem \cite{EHS} (or \cite[Theorem 2.2]{Ell-order}), 
there exists a simple unital  monotracial AF algebra $B$ such that 
$$
K_0(B)=\left\{a_0+\sum_{n=1}^{m}a_nr^{\frac{1}{2^{n}}}\in\mathbb{R}\; 
|\; m\in\mathbb{N}, a_0, a_1,..., a_m\in\mathbb{Q}\right\}, 
$$ 
$K_0(B)_{+}=K_0(B)\cap \mathbb{R}_{+}$ and $[1]_0=1$. 
For any $n\in\mathbb{N}$, there exist mutually orthogonal projections 
$p_{1,n}$,..., $p_{k, n}$ in $B$ such that $\sum_{j=1}^{k}p_{j, n}=1$, 
$\tau_{B}(p_{k, n})=\frac{1+(k-1)r^{\frac{1}{2^n}}}{k}$ and  
$\tau_{B}(p_{j, n})=\frac{1-r^{\frac{1}{2^n}}}{k}$ for any $1\leq j \leq k-1$ 
because we have 
$\frac{1+(k-1)r^{\frac{1}{2^n}}}{k}\in K_0(B)_{+}\cap (\frac{1}{k},1)$, 
$\frac{1-r^{\frac{1}{2^n}}}{k}\in K_0(B)_{+}\cap (0,\frac{1}{k})$ 
and 
$\frac{1+(k-1)r^{\frac{1}{2^n}}}{k}+(k-1)\times \frac{1-r^{\frac{1}{2^n}}}{k}=1$.
For any $n\in\mathbb{N}$, put 
$$
u_n:= \sum_{j=1}^{k}\zeta_{k}^{j}p_{j,n}\in B.
$$
Then $u_n$ is a unitary element such that $u_n^{k}=1$ and we have 
\begin{align*}
\tau_{B}(u_n^l)
&= \tau_{B}\left(\sum_{j=1}^{k}\zeta_{k}^{lj}p_{j,n}\right)
= \sum_{j=1}^{k}\zeta_k^{lj}\tau_{B}(p_{j,n}) 
= \sum_{j=1}^{k-1}\zeta_k^{lj} \times \frac{1-r^{\frac{1}{2^n}}}{k}+ \frac{1+(k-1)r^{\frac{1}{2^n}}}{k} \\ 
&=-\frac{1-r^{\frac{1}{2^n}}}{k} + \frac{1+(k-1)r^{\frac{1}{2^n}}}{k}=r^{\frac{1}{2^n}}
\end{align*}
for any $1\leq l \leq k-1$. Note that $\zeta_k^{l}$ is a root of unity and not equal to $1$. 
Let $A:=\bigotimes_{n=1}^{\infty} B$. 
Then $A$ is a simple unital  monotracial AF algebra. Note that the unique tracial state $\tau_{A}$ 
is a product of traces $\tau_{B}$ in each component. 
For any $n\in\mathbb{N}$, put 
$$
w_n:= u_1\otimes u_2\otimes \cdots \otimes u_n\otimes 1 \otimes \cdots \in  
\bigotimes_{n=1}^{\infty} B=A.
$$
Then $w_n$ is a unitary element such that $w_n^k=1$. 
We shall show that $\{w_n\}_{n\in\mathbb{N}}$ is a Cauchy sequence with respect to the 2-norm. 
Let $\varepsilon>0$. Take a natural number $N$ such that 
$\sum_{j=n+1}^{m}\frac{1}{2^{j}}<\log_{r}{(1-\varepsilon)}$ for any $m> n\geq N$. 
Then we have 
\begin{align*}
\tau_{A}((w_{m}-w_{n})^*(w_{m}-w_{n}))
&=2-2\mathrm{Re} \tau_{A}(1\otimes \cdots \otimes 1\otimes u_{n+1}\otimes \cdots \otimes u_{m}
\otimes 1\cdots) \\
&=2- 2\prod_{j=n+1}^{m}r^{\frac{1}{2^j}}
=2- 2r^{\sum_{j=n+1}^{m}\frac{1}{2^j}}< 2\varepsilon
\end{align*}
for any $m> n\geq N$. 
Hence there exists a unitary element $V$ in 
$\pi_{\tau_{A}}(A)^{''}$ such that $\{w_n\}_{n\in\mathbb{N}}$ converges to $V$ in the strong-$^*$ 
topology. 
We have $V^{k}=1$ and 
\begin{align*}
\tilde{\tau}_{A}(V^l)= \lim_{n\to\infty}\tau_{A}(w_n^l)=\lim_{n\to \infty}\tau_A(u_1^l)\tau_{A}(u_2^l)\cdots
\tau_{A}(u_n^l)=\prod_{n=1}^{\infty}r^{\frac{1}{2^n}}=r
\end{align*}
for any $1\leq l\leq k-1$.
It is easy to see that $\mathrm{Ad}(V)$ induces an 
action $\alpha$ of $\mathbb{Z}_k$ on $A$. 
Note that we have 
$$
\alpha_{[l]} (x_1\otimes x_2\otimes \cdots \otimes x_n\otimes 1 \otimes \cdots)
=u_1^{l}x_1u_1^{*l}\otimes u_2^{l}x_2u_2^{*l}\otimes \cdots \otimes u_n^{l}x_nu_n^{*l}\otimes 1 
\otimes \cdots
$$
for any $1\leq l \leq k$. 
We shall show that $\alpha$ is outer. Let $l$ be a natural number with $1\leq l \leq k-1$. 
Since we have $\tau_{B}(p_{k,n})>\tau_{B}(p_{1,n})$ for any $n\in\mathbb{N}$, 
there exists a partial isometry 
$s_n$ in $B$ such that $s_ns_n^*=p_{1,n}$ and $s_n^*s_n \leq p_{k,n}$. 
Put 
$$
x_n:= 
\overbrace{1\otimes \cdots \otimes 1}^{n-1}\otimes s_n \otimes 1\cdots 
$$
Then $(x_n)_n$ is a central sequence in $A$ of norm one. 
Since we have $u_n^{l}s_nu_n^{*l}=\zeta_k^{l}s_n$, 
$\alpha_{[l]}(x_n)=\zeta_k^{l}x_n$
for any $n\in\mathbb{N}$.
Therefore $\alpha_{[l]}$ induces a non-trivial automorphism of $F(A)$. 
This implies $\alpha_{[l]}$ is not an inner automorphism of $A$. 
Consequently, $\alpha$ is outer. 
\end{proof}

The following theorem is the main result in this section. 

\begin{thm}
Let $\Gamma$ be a finite abelian group, and let $N$ be a subgroup of $\Gamma$. 
Suppose that $m$ is a  probability measure with full support on the set $P_{N}$ 
of minimal projections in $\mathbb{C}N$. 
Then there exist a simple unital monotracial 
AF algebra $A_{(\Gamma, N, m)}$ and an outer action 
$\alpha^{(\Gamma, N, m)}$ on $A_{(\Gamma, N, m)}$ such that the characteristic invariant of 
$\tilde{\alpha}^{(\Gamma, N, m)}$ is trivial, $N(\tilde{\alpha}^{(\Gamma, N, m)})=N$ and 
$i(\tilde{\alpha}^{(\Gamma, N, m)})=[m]$. 
\end{thm}
\begin{proof}
Since $\Gamma$ is a finite abelian group, we may assume that there exist a natural number $n$ and 
prime powers $k_1$,..., $k_{n}$ such that $\Gamma=\bigoplus_{j=1}^{n} \mathbb{Z}_{k_j}$. 
For any $0\leq i \leq n$, let 
$$
N_{i}:=\left\{([l_j])_{j=1}^n\in N\subseteq \bigoplus_{j=1}^{n} \mathbb{Z}_{k_j}\; |\; 
i=|\{j\in\{1,...,n\}\; |\; [l_j]\neq [0]\}|
\right\}.
$$
Then we have $N=\bigsqcup_{i=0}^n N_{i}$ (which is a disjoint union) and $N_i^{-1}=N_{i}$ for any 
$0\leq i \leq n$. We identify $P_{N}$ with 
$\left\{p_{\eta}=\frac{1}{|N|}\sum_{h\in N} \eta (h)h\; |\; \eta\in \hat{N}\right\}$. 
Since $m$ has full support, there exists a real number $r$ with $0<r<1$ such that 
$$
\frac{1}{|N|}\sum_{j=0}^{n}\frac{1}{r^j}\sum_{h\in N_j}\sum_{\eta^{\prime}\in \hat{N}}\eta(h)
\overline{\eta^{\prime}(h)}m(p_{\eta^{\prime}})\geq 0
$$
for any $\eta\in \hat{N}$. Indeed, we have 
\begin{align*}
\lim_{r\to 1-0}\frac{1}{|N|}\sum_{j=0}^{n}\frac{1}{r^j}\sum_{h\in N_j}\sum_{\eta^{\prime}\in 
\hat{N}}\eta(h)\overline{\eta^{\prime}(h)}m(p_{\eta^{\prime}})
&= \frac{1}{|N|}\sum_{h\in N}\sum_{\eta^{\prime}\in 
\hat{N}}\eta(h)\overline{\eta^{\prime}(h)}m(p_{\eta^{\prime}}) \\
&=\frac{1}{|N|}\sum_{\eta^{\prime}\in 
\hat{N}}m(p_{\eta^{\prime}})\sum_{h\in N}\eta\eta^{\prime-1}(h) \\
&= \frac{1}{|N|}\times m(p_{\eta})\times |N| 
=m(p_{\eta})>0
\end{align*}
and
\begin{align*}
\overline{\frac{1}{|N|}\sum_{j=0}^{n}\frac{1}{r^j}\sum_{h\in N_j}\sum_{\eta^{\prime}\in \hat{N}}\eta(h)
\overline{\eta^{\prime}(h)}m(p_{\eta^{\prime}})}
&=\frac{1}{|N|}\sum_{j=0}^{n}\frac{1}{r^j}\sum_{h\in N_j}\sum_{\eta^{\prime}\in \hat{N}}
\overline{\eta(h)}\eta^{\prime}(h)m(p_{\eta^{\prime}}) \\
&= \frac{1}{|N|}\sum_{j=0}^{n}\frac{1}{r^j}\sum_{h\in N_j}\sum_{\eta^{\prime}\in \hat{N}}
\eta(h^{-1})\overline{\eta^{\prime}(h^{-1})}m(p_{\eta^{\prime}}) \\
&= \frac{1}{|N|}\sum_{j=0}^{n}\frac{1}{r^j}\sum_{h\in N_j}\sum_{\eta^{\prime}\in \hat{N}}\eta(h)
\overline{\eta^{\prime}(h)}m(p_{\eta^{\prime}})
\end{align*}
for any $\eta\in \hat{N}$.
Hence a sufficiently large $r<1$ satisfies the property above. 
Define a map $m^{\prime}$ from $P_{N}$ to $\mathbb{R}_{+}$ by 
$$
m^{\prime}(p_{\eta})
=\frac{1}{|N|}\sum_{j=0}^{n}\frac{1}{r^j}\sum_{h\in N_j}\sum_{\eta^{\prime}\in \hat{N}}\eta(h)
\overline{\eta^{\prime}(h)}m(p_{\eta^{\prime}})
$$
for any $\eta\in \hat{N}$. 
Since we have 
\begin{align*}
\sum_{\eta\in \hat{N}} m^{\prime}(p_{\eta})
&=\frac{1}{|N|}\sum_{\eta\in \hat{N}}\sum_{j=0}^{n}\frac{1}{r^j}
\sum_{h\in N_j}\sum_{\eta^{\prime}\in \hat{N}}\eta(h)
\overline{\eta^{\prime}(h)}m(p_{\eta^{\prime}}) \\
&=\frac{1}{|N|}\sum_{j=0}^{n}\frac{1}{r^j}
\sum_{h\in N_j}\sum_{\eta^{\prime}\in \hat{N}}
\overline{\eta^{\prime}(h)}m(p_{\eta^{\prime}})\sum_{\eta\in \hat{N}}\eta(h) \\
&= \frac{1}{|N|}\sum_{\eta^{\prime}\in \hat{N}}
\overline{\eta^{\prime}(\iota)}m(p_{\eta^{\prime}})\times |N|=
\sum_{\eta^{\prime}\in \hat{N}}m(p_{\eta^{\prime}})
=1,
\end{align*}
$m^{\prime}$ is a probability measure on $P_{N}$. 

Lemma \ref{lem:jones} implies that there exist a simple unital monotracial AF algebra $B_0$, 
an action $\beta^{(0)}$ on $B_0$ of $\Gamma$ 
and a unitary representation $u$ of $N$ on $B_0$ such that $N(\beta)=N(\tilde{\beta})=N$, 
$\beta^{(0)}_h=\mathrm{Ad}(u_h)$, $\beta^{(0)}_g(u_h)=u_h$ for any $h\in N$ and $g\in \Gamma$ 
and $\tau_{B_0}(\Phi_{u}(p_{\eta}))=m^{\prime}(p_{\eta})$ for any $\eta\in \hat{N}$. 
Note that we have 
\begin{align*}
\sum_{\eta\in\hat{N}}\overline{\eta(h)}\Phi_{u}\left(p_{\eta}\right)
= \frac{1}{|N|}\sum_{\eta\in\hat{N}}\sum_{h^{\prime}\in N} \overline{\eta(h)}
\eta (h^{\prime})u_{h^{\prime}}
= \frac{1}{|N|}\sum_{h^{\prime}\in N}u_{h^{\prime}}\sum_{\eta\in\hat{N}}\eta(h^{-1}h^{\prime})=u_{h}
\end{align*}
for any $h\in N$. 
Hence if $h$ is an element in $N_{i}$ for some $0\leq i \leq n$, then 
\begin{align*}
\tau_{B_0}(u_{h})&=
\tau_{B_0}\left(\sum_{\eta\in\hat{N}}\overline{\eta(h)}\Phi_{u}\left(p_{\eta}\right)\right)
=\sum_{\eta\in \hat{N}}\overline{\eta(h)}m^{\prime}(p_{\eta}) \\
&= \frac{1}{|N|}\sum_{\eta\in\hat{N}}
\sum_{j=0}^{n}\frac{1}{r^j}\sum_{h^{\prime}\in N_j}\sum_{\eta^{\prime}\in \hat{N}}
\overline{\eta(h)}\eta(h^{\prime})\overline{\eta^{\prime}(h^{\prime})}m(p_{\eta^{\prime}}) \\
&= \frac{1}{|N|}\sum_{j=0}^{n}\frac{1}{r^j}\sum_{h^{\prime}\in N_j}\sum_{\eta^{\prime}\in \hat{N}}
\overline{\eta^{\prime}(h^{\prime})}m(p_{\eta^{\prime}})
\sum_{\eta\in\hat{N}}\eta(h^{-1}h^{\prime}) \\
&= \frac{1}{|N|} \times \frac{1}{r^i}\sum_{\eta^{\prime}\in \hat{N}}
\overline{\eta^{\prime}(h)}m(p_{\eta^{\prime}})\times |N|
=\frac{1}{r^i}\sum_{\eta^{\prime}\in \hat{N}}\overline{\eta^{\prime}(h)}m(p_{\eta^{\prime}}).
\end{align*}
For any natural number $j$ with $1\leq j\leq n$, 
there exist a simple unital monotracial 
AF algebra $B_j$ and a unitary element $V_j$ in 
$\pi_{\tau_{B_j}}(B_j)^{''}$ such that $\mathrm{Ad}(V_j)$ induces an 
outer action of $\mathbb{Z}_{k_j}$ on $B_j$, $V_{j}^{k_j}=1$ 
and $\tilde{\tau}_{B_j}(V_j^{l})=r$ for any $1\leq l \leq k_j-1$ 
by Lemma \ref{lem:outer}.  
Let $\beta^{(j)}$ be the induced action by $\mathrm{Ad}(V_j)$ on $B_j$ of $\mathbb{Z}_{k_j}$. 
Put 
$$
A_{(\Gamma, N, m)}:= \bigotimes_{j=0}^n B_j 
$$
and define an action $\alpha^{(\Gamma, N, m)}$ on $A_{(\Gamma, N, m)}$ of $\Gamma$ by 
$$
\alpha^{(\Gamma, N, m)}_{g}:=\beta^{(0)}_{g}\otimes 
\bigotimes_{j=1}^{n} \beta^{(j)}_{[l_j]}
$$
for any $g=([l_j])_{j=1}^{n}\in \Gamma=\bigoplus_{j=1}^{n} \mathbb{Z}_{k_j}$. 
Then $A_{(\Gamma, N, m)}$ is a simple unital monotracial AF algebra. 
We denote by $\tau$ the unique tracial state on $A_{(\Gamma, N, m)}$. 
It can be easily checked that $N(\alpha^{(\Gamma, N, m)})=\{\iota\}$ and 
$N(\tilde{\alpha}^{(\Gamma, N, m)})=N$ since we have $N(\beta^{(j)})=\{\iota \}$ for any 
$1\leq j\leq n$ and $N(\tilde{\beta}^{(0)})=N$. 

Define a unitary representation $v$ of $N$ on 
$\pi_{\tau}( A_{(\Gamma, N, m)})^{''}$ by 
$$
v_{h}:=\pi_{\tau_{B_0}}(u_{h})\otimes \bigotimes_{j=1}^nV_j^{l_j}
\in \pi_{\tau_{B_0}}(B_0)^{''}\otimes \bigotimes_{j=1}^n 
\pi_{\tau_{B_j}}(B_j)^{''} \cong \pi_{\tau}( A_{(\Gamma, N, m)})^{''}
$$ 
for any $h=([l_j])_{j=1}^n\in N$. 
Then $\tilde{\alpha}^{(\Gamma, N, m)}_{h}=\mathrm{Ad}(v_{h})$ and 
$\tilde{\alpha}^{(\Gamma, N, m)}_{g}(v_h)=v_h$ for any $h\in N$ and $g\in \Gamma$. 
This implies that the characteristic invariant of $\tilde{\alpha}^{(\Gamma, N, m)}$ is trivial. 
Since we have 
\begin{align*}
\tilde{\tau} (\Phi_{v}(p_{\eta}))
&=\frac{1}{|N|}\sum_{h\in N}\eta(h) \tilde{\tau} (v_h) \\
&
=\frac{1}{|N|}\sum_{h=([l_j])_{j=1}^n\in N}\eta(h) \tau_{B_0} (u_{h})
\times \prod_{j=1}^{n}\tilde{\tau}_{B_j}(V_j^{l_j}) \\
&= \frac{1}{|N|}\sum_{j=0}^n\sum_{h\in N_{j}}\eta(h)\times 
\frac{1}{r^j}\sum_{\eta^{\prime}\in \hat{N}}\overline{\eta^{\prime}(h)}m(p_{\eta^{\prime}})\times r^{j} \\
&= \frac{1}{|N|}\sum_{h\in N}\sum_{\eta^{\prime}\in \hat{N}}\eta\eta^{\prime-1}(h)m(p_{\eta^{\prime}}) \\
&= \frac{1}{|N|}\sum_{\eta^{\prime}\in \hat{N}}m(p_{\eta^{\prime}})\sum_{h\in N}\eta\eta^{\prime-1}(h)\\
&=\frac{1}{|N|}\times m(p_{\eta})\times |N| =m(p_{\eta}),
\end{align*}
$i(\tilde{\alpha}^{(\Gamma, N, m)})=[m]$. Consequently, the proof is complete. 
\end{proof}

The following corollary is an immediate consequence of
Proposition \ref{pro:realized-invariant}, Corollary \ref{main:cor} and the theorem above. 

\begin{cor}
Let $A$ be a simple separable nuclear monotracial C$^*$-algebra, and let $\alpha$ be an 
outer action of a finite abelian group $\Gamma$ on $A$. 
Assume that the characteristic invariant of $\tilde{\alpha}$ is trivial. Then there exists a 
probability measure $m$ with full support on $P_{N(\tilde{\alpha})}$ such that 
$\alpha\otimes\mathrm{id}_{\mathcal{W}}$ on $A\otimes \mathcal{W}$ is conjugate to 
$\alpha^{(\Gamma, N(\tilde{\alpha}), m)}\otimes \mathrm{id}_{\mathcal{W}}$ on 
$A_{(\Gamma, N(\tilde{\alpha}), m)}\otimes \mathcal{W}$. 
\end{cor}

\section*{Acknowledgments}
The author would like to thank Eusebio Gardella for pointing out a misleading terminology.

\end{document}